\pdfoutput=1 % let arxiv.org use pdflatex, to process tikz pictures correctly

\documentclass[11pt]{amsart}
\usepackage[
  paper=a4paper,
  text={155mm,230mm},
  centering
]{geometry}

\usepackage{amssymb,amsthm}

\usepackage{enumitem}
\setlist[enumerate]{label=\textup{(\arabic*)}}

\usepackage[color=yellow!90!black]{todonotes}

\usepackage{tikz}
\usetikzlibrary{cd,calc}
\tikzcdset{arrow style=math font}

\usepackage[bookmarks]{hyperref}
\hypersetup{colorlinks=true,linkcolor=blue,citecolor=blue}

\iftrue
\makeatletter
\def\@settitle{%
  \vspace*{-20pt}
  \begin{flushleft}%
    \baselineskip14\p@\relax
    \normalfont\bfseries\LARGE
    \@title
  \end{flushleft}%
}
\def\@setauthors{%
  \begingroup
  \def\thanks{\protect\thanks@warning}%
  \trivlist
  \large \@topsep30\p@\relax
  \advance\@topsep by -\baselineskip
  \item\relax
  \author@andify\authors
  \def\\{\protect\linebreak}%
  \authors
  \ifx\@empty\contribs
  \else
    ,\penalty-3 \space \@setcontribs
    \@closetoccontribs
  \fi
  \normalfont
  \endtrivlist
  \endgroup
}
\def\@setaddresses{\par
  \nobreak \begingroup\raggedright
  \small
  \def\author##1{\nobreak\addvspace\smallskipamount}%
  \def\\{\unskip, \ignorespaces}%
  \interlinepenalty\@M
  \def\address##1##2{\begingroup
    \par\addvspace\bigskipamount\noindent
    \@ifnotempty{##1}{(\ignorespaces##1\unskip) }%
    {\ignorespaces##2}\par\endgroup}%
  \def\curraddr##1##2{\begingroup
    \@ifnotempty{##2}{\nobreak\noindent\curraddrname
      \@ifnotempty{##1}{, \ignorespaces##1\unskip}\/:\space
      ##2\par}\endgroup}%
  \def\email##1##2{\begingroup
    \@ifnotempty{##2}{\smallskip\nobreak\noindent E-mail address%
      \@ifnotempty{##1}{, \ignorespaces##1\unskip}\/:\space
      \ttfamily##2\par}\endgroup}%
  \def\urladdr##1##2{\begingroup
    \def~{\char`\~}%
    \@ifnotempty{##2}{\nobreak\noindent\urladdrname
      \@ifnotempty{##1}{, \ignorespaces##1\unskip}\/:\space
      \ttfamily##2\par}\endgroup}%
  \addresses
  \endgroup
  \global\let\addresses=\@empty
}
\def\@setabstracta{%
  \ifvoid\abstractbox
  \else
    \skip@20\p@ \advance\skip@-\lastskip
    \advance\skip@-\baselineskip \vskip\skip@
    \box\abstractbox
    \prevdepth\z@ % because \abstractbox is a \vtop
  \fi
}
\def\@maketitle{%
  \normalfont\normalsize
  \@adminfootnotes
  \@mkboth{\@nx\shortauthors}{\@nx\shorttitle}%
  \global\topskip42\p@\relax % 5.5pc   "   "   "     "     "
  \@settitle
  \ifx\@empty\authors \else \@setauthors \fi
  \ifx\@empty\@dedicatory
  \else
    \baselineskip18\p@
    \vtop{\centering{\footnotesize\itshape\@dedicatory\@@par}%
      \global\dimen@i\prevdepth}\prevdepth\dimen@i
  \fi
  \@setabstract
  \normalsize
  \if@titlepage
    \newpage
  \else
    \dimen@20\p@ \advance\dimen@-\baselineskip
    \vskip\dimen@\relax
  \fi
} % end \@maketitle

\let\oldtocsection=\tocsection
\let\oldtocsubsection=\tocsubsection
\let\oldtocsubsubsection=\tocsubsubsection
\renewcommand{\tocsection}[2]{\hspace{0em}\oldtocsection{#1}{#2}}
\renewcommand{\tocsubsection}[2]{\hspace{2em}\oldtocsubsection{#1}{#2}}
\renewcommand{\tocsubsubsection}[2]{\hspace{4em}\oldtocsubsubsection{#1}{#2}}
\renewenvironment{abstract}{%
  \ifx\maketitle\relax
    \ClassWarning{\@classname}{Abstract should precede
      \protect\maketitle\space in AMS document classes; reported}%
  \fi
  \global\setbox\abstractbox=\vtop \bgroup
    \normalfont\small
    \list{}{\labelwidth\z@
      \leftmargin0pc \rightmargin\leftmargin
      \listparindent\normalparindent \itemindent\z@
      \parsep\z@ \@plus\p@
      
    }%
    \item[\hskip\labelsep\bfseries\abstractname.]%
}{%
  \endlist\egroup
  \ifx\@setabstract\relax \@setabstracta \fi
}
\def\section{\@startsection{section}{1}%
  \z@{-1.2\linespacing\@plus-.5\linespacing}{.8\linespacing}%
  {\normalfont\bfseries\large}}
\def\subsection{\@startsection{subsection}{2}%
  \z@{-.8\linespacing\@plus-.3\linespacing}{.5\linespacing\@plus.3\linespacing}%
  {\normalfont\bfseries}}
\def\subsubsection{\@startsection{subsubsection}{3}%
  \z@{.7\linespacing\@plus.1\linespacing}{-1.5ex}%
  {\normalfont\bfseries}}
\def\@secnumfont{\bfseries}
\makeatother
\fi %\iftrue/false for amsart.cls hack

\makeatletter
\@namedef{subjclassname@2020}{%
  \textup{2020} Mathematics Subject Classification}
\makeatother

\makeatletter
\def\@newl@bel#1#2#3{{%
  \@ifundefined{#1@#2}%
    {\global\@namedef{#1@#2}{#3}}
    {\gdef \@multiplelabels {%
       \@latex@warning@no@line{There were multiply-defined labels}}%
     \@latex@warning@no@line{Label `#2' multiply defined}}%
  }}
\makeatother

\makeatletter
\def\Nopagebreak{\@nobreaktrue\nopagebreak}
\makeatother

\foreach \n in % \mathbb
{Z,Q,R,C}
{\expandafter\xdef\csname\n\endcsname{\noexpand\mathbb{\n}}}

\foreach \n in % \mathbb
{A,F}
{\expandafter\xdef\csname b\n\endcsname{\noexpand\mathbb{\n}}}

\foreach \n in % \mathcal
{A,...,Z}
{\expandafter\xdef\csname c\n\endcsname{\noexpand\mathcal{\n}}}

\foreach \n in % \operatorname
{dim,sign,sgn,Ker,Coker,Im,Hom,End,GL,SO,
 Tor,Ext,Gal,Irr,dis,Wh,tr,Tr,vol,colim,lk,st}
{\expandafter\xdef\csname\n\endcsname{\noexpand\operatorname{\n}}}

\foreach \n/\s in % \operatorname with distinct macro name
{inte/int}
{\expandafter\xdef\csname\n\endcsname{\noexpand\operatorname{\s}}}

\foreach \n in % \textrm
{id}
{\expandafter\xdef\csname\n\endcsname{\noexpand\textrm{\n}}}

\def\sm{\smallsetminus}

\def\cupover#1{\mathbin{\mathop{\cup}\limits_{#1}}}

\def\cl{\operatorname{cl}}
\def\scl{\operatorname{scl}}

\makeatletter
\def\to{\mathchoice{\longrightarrow}{\rightarrow}{\rightarrow}{\rightarrow}}
\newcommand{\shortxra}[2][]{\ext@arrow 0359\rightarrowfill@{#1}{#2}}
\def\longrightarrowfill@{\arrowfill@\relbar\relbar\longrightarrow}
\newcommand{\longxra}[2][]{\ext@arrow 0359\longrightarrowfill@{#1}{#2}}

\makeatother

\makeatletter
\newtheoremstyle{theorem-giventitle}
        {}{}                      %%% space between body and thm
        {\itshape}                %%% Thm body font
        {}                        %%% Indent amount (empty = no indent)
        {\bfseries}               %%% Thm head font
        {.}                       %%% Punctuation after thm head
        {\thm@headsep}            %%% Space after thm head
        {\thmnote{\bfseries#3}}   %%% Thm head spec
\newtheoremstyle{theorem-givenlabel}
        {}{}                      %%% space between body and thm
        {\itshape}                %%% Thm body font
        {}                        %%% Indent amount (empty = no indent)
        {\bfseries}               %%% Thm head font
        {.}                       %%% Punctuation after thm head
        {\thm@headsep}            %%% Space after thm head
        {\thmname{#1}~\thmnumber{#3}\setcurrentlabel{#3}}%%% Thm head spec
\newtheoremstyle{definition-giventitle}
        {}{}                      %%% space between body and thm
        {}                        %%% Thm body font
        {}                        %%% Indent amount (empty = no indent)
        {\bfseries}               %%% Thm head font
        {.}                       %%% Punctuation after thm head
        {\thm@headsep}            %%% Space after thm head
        {\thmnote{\bfseries#3}}   %%% Thm head spec
\def\setcurrentlabel#1{\gdef\@currentlabel{#1}}
\makeatother

\theoremstyle{plain}
\newtheorem{theorem}{Theorem}[section]
\newtheorem{theoremalpha}{Theorem}

\newtheorem{proposition}[theorem]{Proposition}
\newtheorem{lemma}[theorem]{Lemma}
\newtheorem{corollary}[theorem]{Corollary}

\newtheorem{question}[theorem]{Question}

\theoremstyle{definition}
\newtheorem{definition}[theorem]{Definition}
\newtheorem{remark}[theorem]{Remark}

\theoremstyle{definition}
\newtheorem*{assertion}{Assertion}

\theoremstyle{theorem-giventitle}
\newtheorem{theorem-named}{}
\theoremstyle{theorem-givenlabel}
\newtheorem{theorem-labeled}{Theorem}

\theoremstyle{definition-giventitle}
\newtheorem{definition-named}{}
\newtheorem{conjecture-named}{}
\newtheorem{case-named}{}
\newtheorem{step-named}{}

\numberwithin{equation}{section}

\begin{document}

\title
{Quantitative bordism over acyclic groups and Cheeger-Gromov $\rho$-invariants}


\author{Jae Choon Cha}
\address{Center for Research in Topology and Department of Mathematics, POSTECH, Pohang 37673, Republic of Korea}
\curraddr{}
\email{jccha@postech.ac.kr}

\author{Geunho Lim}
\address{Institut de Mathématiques de Jussieu-Paris Rive Gauche, CNRS, Sorbonne Université, Université Paris Cité, 4 Place Jussieu, Paris, 75005, France}
\curraddr{}
\email{lim@imj-prg.fr}

% \thanks{This work was partially supported by the National Research Foundation of Korea grant 2019R1A3B2067839.
% The second author was also partially supported by Horizon Europe ERC Grant number:101045750/Project acronym:HodgeGeoComb.}

\begin{abstract}
  We obtain a solution to a bordism version of Gromov's linearity problem over a large family of acyclic groups, for manifolds with arbitrary dimension.
  Every group embeds into some acyclic group in this family.
  Thus, the linear bordism problem has an affirmative solution over a possibly enlarged acyclic group.
  Our result holds in both PL and smooth categories, and for both oriented and unoriented cases.
  In the PL case, our results hold without assuming bounded local geometry.
  As an application, we prove that there is a universal linear bound for the Cheeger-Gromov $L^2$ $\rho$-invariants of PL $(4k-1)$-manifolds associated with arbitrary regular covers.
  We also show that the minimum number of simplices in a PL triangulation of $(4k-1)$-manifolds with a fixed simple homotopy type is unbounded if the fundamental group has nontrivial torsion.
  The proof of our main results builds on quantitative algebraic and geometric techniques over the simplicial classifying spaces of groups.
\end{abstract}

\maketitle

\section{Introduction and main results}

In his 1999 article~\cite{Gromov:1999-1}, Gromov introduced a quantitative viewpoint of the study of topology with several fundamental questions.
The underlying motivation is that even when algebraic topology successfully determines the existence of a solution to a given problem, if one attempts to understand the nature of a solution, for instance, its ``size'' or ``complexity,'' interesting new questions often arise.

Specifically, Gromov considered cobordism theory.
Thom's work tells us when a closed $n$-manifold bounds a null-cobordism, but it presents little about the geometric intricacy of a null-cobordism.
For a rigorous formulation, define the \emph{complexity} $V(M)$ of a smooth manifold $M$ to be the infimum of the volume of a Riemannian structure on $M$ with bounded local geometry, i.e.,\ $($injectivity radius$) \ge 1$ and $|$sectional curvature$|\le 1$.
(When $\partial M$ is nonempty, we require that the Riemannian structure restricts to a product on a collar $\partial M\times I$ where $I$ has unit length.)
Gromov asked about the existence of a linear complexity cobordism~\cite{Gromov:1999-1}:
\emph{if a smooth closed $n$-manifold $M$ is null-cobordant, is there a smooth null-cobordism $W$ such that $V(W) \le C(n)\cdot V(M)$ where $C(n)$ is a universal constant depending only on $n=\dim M$?}

A PL version of Gromov's linearity problem can be formulated by using the following definition in place of~$V(-)$.

\begin{definition}\label{definition:PL-complexity}
  For a PL manifold, $M$, we define the \emph{complexity}, $\Delta(M)$, to be the minimal number of simplices in a PL triangulations of~$M$.
\end{definition}

The linear cobordism problem is still open for both smooth and PL cases.
In the smooth case, Chambers, Dotterrer, Manin and Weinberger~\cite{Chambers-Dotterrer-Manin-Weinberger:2018-1} showed that there is a null-cobordism $W$ with complexity arbitrarily close to linear:
$V(W) \in O(V(M)^{1+\epsilon})$ where $\epsilon>0$ is given arbitrary.
For the PL case, Manin and Weinberger proved a similar almost linear result under an additional assumption on local geometry (see Remark~\ref{remark:bounded-geometry-condition} below) in a recent paper~\cite{Manin-Weinberger:2023-1}.

Replacing cobordism with bordism leads to a natural generalization of the problem.
In this paper, we focus on bordism over a group~$G$.
A \emph{manifold over $G$} is a manifold $M$ equipped with a map $M\to BG$, where $BG$ is the classifying space of the group~$G$.
For two manifolds $M$ and $N$ over $G$, a bordism from $M$ to $N$ over $G$ is a compact manifold $W$ over $G$ cobounded by $M$ and $N$ such that the map $W\to BG$ restricts to the given maps $M\to BG$ and $N\to BG$.
If $N$ is empty, $W$ is a null-bordism over~$G$.
A cobordism is a bordism over a trivial group.
A direct analog of the above linear null-cobordism question of Gromov is the following null-bordism question: \emph{if $M$ is null-bordant over $G$, does there exist a null-bordism $W$ over $G$ with $\Delta(W) \le C(n) \cdot \Delta(M)$?}

To formulate a refined version \emph{modulo contributions from cobordism}, we consider a bordism $W$ over $G$ from a given manifold $M$ to another manifold $N$ that is over $G$ via a constant map $N\to BG$.
\[
  \begin{tikzcd}
    M \ar[r,hook] \ar[rd] & W \ar[d] & N \ar[l,hook'] \ar{d}
    \\
    & BG & \{*\} \ar{l}
  \end{tikzcd}
\]
We call such $W$ a \emph{bordism from $M$ to a trivial end}, following~\cite[Definition~3.1]{Cha:2014-1}.
For example, let $G=\Z$ and $M=S^1\times S^1$, equipped with the projection $M\to S^1 = BG$ onto the first factor.
Attach a 2-handle to $M\times I$ along the circle $\{\text{pt}\} \times S^1 \times 1 \subset M\times 1$ to obtain a 3-manifold~$W$.
Then $\partial W=M \cup S^2$, and the map $M\to BG$ extends to $W \to BG$ which restricts to a constant map on~$S^2$.
Thus $W$ is a bordism over $G$ from $M$ to a trivial end.

Note that if Gromov's linear cobordism problem has an affirmative answer, then for a null-cobordant manifold~$M$, a bordism $W$ from $M$ to a trivial end can always be extended to a null-bordism $W'$ over $G$ with $\Delta(W') \le C(n) \cdot \Delta(W)$.
To obtain $W'$, attach to $W$ a null-cobordism with linearly bounded complexity along the trivial end.
Conversely, a null-bordism for $M$ is clearly a bordism to a trivial end.
Thus, the linear null-bordism problem over a group $G$ is equivalent to the combination of the linear null-cobordism problem (over a trivial group) and the following:

\begin{question}[Bordism version of Gromov's linearity problem]
  \label{question:linear-bordism}
  If a closed PL $n$-manifold $M$ over $G$ with $n>1$ admits a PL bordism to a trivial end, then is there a PL bordism $W$ from $M$ to a trivial end over $G$ such that $\Delta(W) \le C(n)\cdot \Delta(M)$, where $C(n)$ is a constant depending only on $n=\dim M$ (independent of $M$, $G$ and $M \to BG$)?
\end{question}

When $G$ is an acyclic group, every closed manifold $M$ over $G$ is bordant to a trivial end, so we have the following special case.

\begin{question}[Question~\ref{question:linear-bordism} for acyclic groups]
  \label{question:linear-bordism-over-acyclic}
  For each closed PL $n$-manifold $M$ over an acyclic group~$G$ with $n>1$,
  is there a PL bordism $W$ from $M$ to a trivial end over $G$ such that $\Delta(W) \le C(n)\cdot \Delta(M)$?
\end{question}

By replacing $\Delta(-)$ by~$V(-)$, the smooth versions of the questions are formulated.
Also, one can consider the oriented and unoriented cases of the questions.
For $n=1$ case of Questions~\ref{question:linear-bordism} and~\ref{question:linear-bordism-over-acyclic}, see Section~\ref{section:questions-for-n=1}.

\subsection{Linear bordism over acyclic groups}

The main result of this paper is an affirmative answer to the linear bordism problem for a large family of acyclic groups.
Our results, stated as Theorems~\ref{theorem:main-quantitative-bordism} and~\ref{theorem:main-quantitative-smooth-bordism} below, immediately imply the following.

\begin{theorem-named}[Corollary to Theorems~\ref{theorem:main-quantitative-bordism} and~\ref{theorem:main-quantitative-smooth-bordism}]
  % \label{corollary:quantitative-bordism-over-enlarged-group}
  If $M$ is a closed manifold over a group $G$, then there is a bordism $W$, over a group $\Gamma$ that contains $G$ as a subgroup, from $M$ to a trivial end with linearly bounded complexity.
  That is, the answer to Question~\ref{question:linear-bordism} is affirmative if $G$ is allowed to be enlarged.
  This holds in both PL and smooth categories and for both oriented and unoriented cases.
\end{theorem-named}

Kan and Thurston showed that every group injects into an acyclic group~\cite{Kan-Thurston:1976-1}.
Subsequently, Baumslag, Dyer and Heller presented a functorial embedding~\cite{Baumslag-Dyer-Heller:1980-1}.
It associates an acyclic group $\cA(G)$ to every group $G$, together with a functorial injection $G\hookrightarrow \cA(G)$.
Theorem~\ref{theorem:main-quantitative-bordism} stated below provides a solution to the linear bordism problem for the Baumslag-Dyer-Heller groups~$\cA(G)$.

Recall from Definition~\ref{definition:PL-complexity} that for a PL manifold $M$, $\Delta(M)$ denotes the minimal number of simplices in a triangulation.
When $M$ is triangulated, i.e., equipped with a PL triangulation, we also denote by $\Delta(M)$ the number of simplices in the given triangulation.
We will use this slight abuse of notation throughout the paper.

\begin{theoremalpha}
  \label{theorem:main-quantitative-bordism}
  Let $M$ be a closed triangulated PL $n$-manifold over the acyclic group~$\cA(G)$, where $G$ is an arbitrary group.
  Then there exists a triangulated PL bordism $W$ over $\cA(G)$ from $M$ to a trivial end such that $\Delta(W) \leq C(n) \cdot \Delta(M)$ where $C(n)$ is a constant depending only on~$n$.
  In addition, if $M$ is oriented, then $W$ is an oriented bordism.
\end{theoremalpha}

From Theorem~\ref{theorem:main-quantitative-bordism}, one readily obtains the PL case of the above corollary.

We discuss some aspects of Theorem~\ref{theorem:main-quantitative-bordism} in the following remarks.

\begin{remark}\sloppy
  \label{remark:bounded-geometry-condition}
  In earlier quantitative studies in the literature (e.g.~\cite{Gromov-Guth:2012-1,Chambers-Dotterrer-Manin-Weinberger:2018-1,Manin-Weinberger:2023-1}), it has been a standard convention to assume a ``bounded (combinatorial) geometry'' condition that there is an upper bound for the number of simplices adjacent to a vertex in a triangulation.
  Our results hold for arbitrary PL triangulations, \emph{without} requiring any bounded geometry hypothesis.
  See Definition~\ref{definition:PL-complexity}.
  This suggests that the answer to the linearity problem may be affirmative for PL manifolds without requiring bounded geometry, as stated in Questions~\ref{question:linear-bordism} and~\ref{question:linear-bordism-over-acyclic}.
\end{remark}

\begin{remark}
  \label{remark:bordism-with-bounded-geometry}
  Theorem~\ref{theorem:main-quantitative-bordism} works in the category of PL manifolds \emph{with} bounded geometry as well.
  For brevity, we say that a simplicial complex has \emph{bounded geometry of type $L$} if each vertex lies in at most $L$ simplices.
  Then in Theorem~\ref{theorem:main-quantitative-bordism}, if $M$ has bounded geometry of type $L$, the bordism $W$ has bounded geometry of type $C'(n)\cdot L$ where $C'(n)$ is a constant depending only on $n=\dim M$.
\end{remark}

\begin{remark}
  \label{remark:PL-null-bordism}
  Assuming bounded geometry, one can enlarge the bordism $W$ from $M$ to a trivial end $N$ in Theorem~\ref{theorem:main-quantitative-bordism} to obtain a null-bordism $W'$ for $M$ with $O(\Delta(M)^{1+\epsilon})$ complexity, by applying Manin-Weinberger's result \cite[Theorem~1.1]{Manin-Weinberger:2023-1} to the trivial end~$N$.
  More precisely, for any fixed $\epsilon>0$, if a given PL $n$-manifold $M$ over $\cA(G)$ has bounded geometry of type $L$, then there is a null-bordism $W'$ over $\cA(G)$ with bounded geometry of type $L'(n,L)$ such that $\Delta(W') \le C(n,L,\epsilon) \cdot \Delta(M)^{1+\epsilon}$.
  Here $L'(n,L)$ and $C(n,L,\epsilon)$ are constants depending only on $(n,L)$ and $(n,L,\epsilon)$ respectively, independent of the manifold $M$ and the target classifying space~$B\cA(G)$.
  This may be compared with results in \cite{Manin-Weinberger:2023-1} that gives an almost linear bound for null-bordism over \emph{finite complexes}~$Y$, where the bound is of the form $C \cdot \Delta(M)^{1+\epsilon}$ with a constant factor $C=C(n,L,\epsilon,Y)$ that depends on the target space~$Y$.
\end{remark}

\subsubsection*{Method of the proof}

The proof of Theorem~\ref{theorem:main-quantitative-bordism} uses two essential ingredients.
They are quantitative techniques in the realm of homological algebra of chain complexes and geometric topology of PL manifolds, respectively.

First, we use a quantitative chain null-homotopy for a given map $M \to B\cA(G)$, where $B\cA(G)$ is the simplicial classifying space of the Baumslag-Dyer-Heller acyclic group~$\cA(G)$.
To obtain this, we depend on results in~\cite{Cha:2014-1}.
See Section~\ref{section:BDH-preliminaries} for further details.

By the Atiyah-Hirzebruch spectral sequence, the existence of such a chain null-homotopy implies that $M$ is bordant to a trivial end.
Our second ingredient is a geometric construction that realizes a given chain null-homotopy as a quantitative bordism over the classifying space $B\Gamma$ of an \emph{arbitrary} discrete group~$\Gamma$.

We remark that a simpler version for 4-dimensional bordism for 3-manifolds $M$ was developed in~\cite{Cha:2014-1}, but the techniques in~\cite{Cha:2014-1} do not directly generalize to high dimensions.
See Remark~\ref{remark:comparison-with-dim-3} for more about this.
Briefly, our results in arbitrary dimensions are obtained as follows.
We develop simplicial transversality techniques and apply them to a given map $M\to B\Gamma^{(p)}$ into the $p$-skeleton of the target classifying space~$B\Gamma$, to obtain submanifolds which are the inverse images of the barycenters of the $p$-simplices of~$B\Gamma$, and more important, PL cobordisms (embedded in $M$) between the submanifolds.
We use those embedded cobordisms to modify $M$ by surgery-like operations, producing a bordism from $M$ to a new manifold that maps to the lower skeleton~$B\Gamma^{(p-1)}$.
Inductively applying this process over the skeleta, we construct a bordism to a trivial end.
Algebraic data from the chain null-homotopy in~\cite{Cha:2014-1} is essential in finding suitable combinations of the geometric surgery-like operations that achieve the above.
We show that the resulting bordism is linear in the size of the algebraic data.

In addition, since the natural cell structure of the simplicial set $B\Gamma$ is not a simplicial complex, we develop techniques for a more general class of  complexes called \emph{simplicial-cell complexes}.
See Sections~\ref{section:simplicial-cellular-transversality} and~\ref{section:quantitative-pl-bordism} for more details.

\begin{remark}
  Our Atiyah-Hirzebruch type approach has some intriguing aspects when it is compared with approaches for quantitative (co)bordism in~\cite{Chambers-Dotterrer-Manin-Weinberger:2018-1,Manin-Weinberger:2023-1} that builds on Thom's theory.

  \begin{enumerate}
    \item Our target space $B\Gamma$, the simplicial classifying space, is an \emph{infinite} complex, while arguments in \cite{Chambers-Dotterrer-Manin-Weinberger:2018-1,Manin-Weinberger:2023-1} work for finite complexes as target spaces.
    Moreover, $B\Gamma$ has \emph{unbounded} local geometry in general. (Recall that our results hold without assuming bounded local geometry.)
    \item The first step in Thom's work is to choose an embedding of a given manifold $M$ in $S^N$, from which one obtains a map of $S^N$ into a Thom space, the target space in this case.
    A key difficulty that prevents linearity in \cite{Chambers-Dotterrer-Manin-Weinberger:2018-1,Manin-Weinberger:2023-1} occurs in this initial step, that is, it is unknown whether there is an efficient embedding for which the Thom map has a linearly bounded Lipschitz constant.
    (The remaining parts of the approach in \cite{Chambers-Dotterrer-Manin-Weinberger:2018-1,Manin-Weinberger:2023-1} can be done linearly.)
    On the other hand, our approach begins with a given $M \to B\Gamma$, and makes use of the fact that for any given triangulation of $M$, it can be realized by a simplicial(-cellular) map $M\to B\Gamma$ \emph{without subdivision}, due to~\cite{Cha:2014-1}.
    (See Section~\ref{section:quantitative-pl-bordism}.)
    This provides linearly controlled initial data for the methods outlined above so that we eventually obtain a linear bordism.
  \end{enumerate}

\end{remark}

\subsubsection*{Smooth quantitative bordism}

By examining how our proof of Theorem~\ref{theorem:main-quantitative-bordism} applies to a piecewise smooth triangulation of a smooth manifold and by smoothing the resulting PL bordism, we obtain a smooth version of Theorem~\ref{theorem:main-quantitative-bordism} which is stated below.
Details of the argument are given in Section~\ref{section:quantitative-smooth-bordism}.

\begin{theoremalpha}
  \label{theorem:main-quantitative-smooth-bordism}
  Let $M$ be a closed smooth $n$-manifold over~$\cA(G)$, with $G$ arbitrary.
  Then there exists a smooth bordism $W$ over $\cA(G)$ from $M$ to a trivial end such that $V(W) \le C(n) \cdot V(M)$, where $C(n)$ is a constant depending only on~$n$.
  In addition, if $M$ is oriented, then $W$ is an oriented bordism.
\end{theoremalpha}

\begin{remark}
  As done in Remark~\ref{remark:PL-null-bordism} for the PL case, by combining Theorem~\ref{theorem:main-quantitative-smooth-bordism} with the main result of~\cite{Chambers-Dotterrer-Manin-Weinberger:2018-1}, we obtain an almost linear smooth null-bordism:
  for any $\epsilon>0$ and $n$, there is a constant $C(n,\epsilon)$ such that every smooth $n$-manifold $M$ over $\cA(G)$ admits a smooth null-bordism $W$ over $\cA(G)$ satisfying $V(W) \le C(n,\epsilon) \cdot V(M)^{1+\epsilon}$.    
\end{remark}

\subsection{Linear bound for Cheeger-Gromov \texorpdfstring{$L^2$}{L2} \texorpdfstring{$\rho$}{rho}-invariants}

In the work of Cheeger and Gromov~\cite{Cheeger-Gromov:1985-1,Cheeger-Gromov:1985-2}, the $L^2$ $\rho$-invariant is defined using a Riemannian metric on a closed oriented $(4k-1)$-manifold~$M$:
for a homomorphism $\phi\colon \pi_1(M)\to G$, they define the invariant $\rho^{(2)}(M,\phi)\in \R$ to be the difference of the analytic $\eta$-invariants of the signature operator and the $L^2$ $\eta$-invariant defined using the von Neumann trace on the regular cover associated with the given homomorphism~$\phi$.
It is known that the invariant can also be defined topologically, as the $L^2$-signature defect of a certain $4k$-manifold $W$ bounding $M$, using an index theoretic approach as appeared in~\cite{Chang-Weinberger:2003-1}.
This shows that the $\rho^{(2)}(M,\phi)$ is a topological invariant, and also it enables us to define the $\rho$-invariant for non-smoothable manifolds.

Cheeger and Gromov showed the following existence of a bound for smooth manifolds, which has been used as a key ingredient in their work~\cite{Cheeger-Gromov:1985-1} and several subsequent applications of the $\rho$-invariant.
For each $M$, there is a constant $C_M$ such that
\begin{equation}
  |\rho^{(2)}(M,\phi)| \leq C_M
  \label{equation:cheeger-gromov-inequality}  
\end{equation}
for all homomorphisms $\phi$ of $\pi_1(M) \to G$ to an arbitrary group~$G$\@.
In fact, considering Riemannian structures on $M$ with bounded local geometry, the bound $C_M$ is of the form $C_M = C(n)\cdot V(M)$, where $C(n)$ depends only on $n=\dim M$.
This was a key motivation for Gromov's linearity conjecture.

While the $\rho$-invariant is a topological invariant, Cheeger-Gromov's original proof of the existence of the above bound was completely analytic~\cite{Cheeger-Gromov:1985-1} (see also~\cite{Ramachandran:1993-1}).
This leads us to a natural question: can one understand the Cheeger-Gromov bound topologically?
Understanding the bound from a topological viewpoint is also important for applications, e.g.,\ see~\cite[Remark~6.6]{Cha:2014-1} for applications to knot concordance initiated by~\cite{Cochran-Teichner:2003-1}, and \cite{Cha:2014-1, Cha:2016-1, Lim-Weinberger:2023-1} for applications to complexity of manifolds.

In~\cite{Cha:2014-1}, a topological approach to the Cheeger-Gromov bound was developed, using the $L^2$-signature defect interpretation, focusing on the case of 3-manifolds.
In particular, the following explicit linear bound was given:
if $M$ is an orientable closed 3-manifold, then for any homomorphism $\phi\colon\pi_1(M)\to G$, $|\rho^{(2)}(M,\phi)| \le 363090 \cdot \Delta(M)$.

By using Theorem~\ref{theorem:main-quantitative-bordism}, we obtain a linear bound in all dimensions, generalizing the above result of~\cite{Cha:2014-1}:

\begin{theoremalpha}\label{theorem:main-rho-invariant} 
  If $M$ is a closed oriented PL $(4k-1)$-manifold, then
  \begin{equation}
    \label{equation:rho-invariant-bound}
    |\rho^{(2)}(M,\phi)| \leq C(4k-1) \cdot \Delta(M)
  \end{equation}
  for every homomorphism $\phi\colon \pi_1(M)\to G$ with $G$ arbitrary, 
  where $C(4k-1)$ is a constant depending only on $4k-1=\dim M$.
\end{theoremalpha}

Also, a Riemannian manifold with bounded local geometry admits a piecewise smooth triangulation with complexity linearly bounded by the volume~\cite[Theorem~3]{Boissonnat-Dyer-Ghosh:2018-1} (see also~\cite[Theorem~C$'$]{Chambers-Dotterrer-Manin-Weinberger:2018-1}).
Thus one obtains the Cheeger-Gromov bound~\eqref{equation:cheeger-gromov-inequality} for smooth manifolds equipped with arbitrary~$\phi$, as a consequence of Theorem~\ref{theorem:main-rho-invariant}\@.

\begin{remark}
  Theorem~\ref{theorem:main-rho-invariant} holds for any PL triangulation, without requiring bounded (combinatorial) geometry, similarly to Theorem~\ref{theorem:main-quantitative-bordism}\@.
\end{remark}

\begin{remark}
  In~\cite{Lim-Weinberger:2023-1}, Weinberger and the second author use the results and techniques developed in this paper to obtain bounds for the classical (non-$L^2$) $\rho$-invariants over finite groups.
  They also develop an independent approach to a bound of $\rho^{(2)}(M,\phi)$ for 
  isomorphisms $\phi\colon \pi_1(M)\to \pi_1(M)$ (consequently for \emph{injections} $\phi\colon \pi_1(M)\hookrightarrow G$ by $L^2$-induction) using a hyperbolization technique in the same paper~\cite{Lim-Weinberger:2023-1}.
  Theorem~\ref{theorem:main-rho-invariant} gives a more general linear bound for arbitrary homomorphisms $\phi\colon \pi_1(M) \to G$.
\end{remark}

\begin{remark}
  The linear bound in Theorem~\ref{theorem:main-rho-invariant} is asymptotically optimal in every dimension:
  if we denote by $B(c)$ the supremum of $|\rho^{(2)}(M,\phi)|$ over all $(M,\phi)$ with $M$ a PL $n$-manifold ($n=4k-1$) such that $\Delta(M) \le c$ and $\phi\colon \pi_1(M)\to G$ arbitrary, then we have
  \[
    C_1(n) \le \limsup_{c\to\infty} \frac{B(c)}{c} \le C_2(n)
  \]
  for some positive constants $C_1(n)$, $C_2(n)$ depending only on~$n$.
  Theorem~\ref{theorem:main-rho-invariant} gives the upper bound~$C_2(n)$.
  The existence of the lower bound $C_1(n) > 0$ is obtained from known values of $\rho^{(2)}$ for lens spaces, following the argument in~\cite[Proof of Theorem~1.6]{Cha:2014-1}.
\end{remark}

\subsubsection*{Applications: complexity bounds from $\rho$-invariants}

Estimating the complexity $\Delta(M)$ of a given PL manifold $M$ is a challenging problem (e.g.\ see~\cite{Jaco-Rubinstein-Tillman:2013-1}).
For an upper bound for the complexity, the problem is essentially to construct an efficient triangulation.
Finding a lower bound is, perhaps arguably, more difficult and interesting.
One can view an inequality of the form in Theorem~\ref{theorem:main-rho-invariant} as a lower bound for the complexity~$\Delta(M)$ given in terms of the $\rho$-invariant.
Applications based on this appeared in~\cite{Cha:2014-1} for the first time, especially for 3-dimensional lens spaces, and then in~\cite{Cha:2016-1} for more general classes of 3-manifolds.
Subsequently, the case of high dimensional lens spaces was considered in~\cite{Lim-Weinberger:2023-1}\@.
These earlier results can also be recovered by using Theorem~\ref{theorem:main-rho-invariant}.

In what follows, we discuss some further applications of the $\rho$-invariants to the study of complexity.
As an example, we address the complexity of connected sum, focusing on the $r$-fold connected sum $\#^r M$ of a closed connected PL $n$-manifold~$M$.
One readily sees that the following subadditivity holds:
\begin{equation}
  \label{equation:sum-subadditivity}
  \Delta(\#^r M) \le r\cdot \Delta(M).  
\end{equation}
Note that the equality in~\eqref{equation:sum-subadditivity} does not hold in general (for instance for $M=S^n$).
It is natural to ask whether there is a lower bound of the form
\begin{equation}
  \label{equation:sum-lower-bound-conjecture}
  \Delta(\#^r M) \ge C(n) \cdot r\cdot \Delta(M)
\end{equation}
for compact $n$-manifolds~$M \ne S^n$, where $C(n)>0$ is a constant.
For $n = 2$, one can see that \eqref{equation:sum-lower-bound-conjecture} holds by using (the classification of compact connected surfaces and) the inequality $\Delta(M) \ge \dim_{\Z_2} H_1(M;\Z_2)$.
To the knowledge of the authors, \eqref{equation:sum-lower-bound-conjecture} is unknown for $n\ge 3$.

Motivated by this, we consider the following stabilized version with respect to the connected sum.
Define the \emph{sum-stable complexity} $\Delta^{\text{st}}(M)$ of a PL manifold $M$ by
\[
  \Delta^{\text{st}}(M) = \lim_{r\to\infty} \frac{\Delta(\#^r M)}{r}.
\]
From the subadditivity~\eqref{equation:sum-subadditivity}, it follows that the limit always exists (see, e.g.,~\cite[Appendix~A]{Livingston:2010-1}) and satisfies $\Delta^{\text{st}}(M) \le \Delta(M)$.
We conjecture that there is a constant $C(n)>0$ such that $\Delta^{\text{st}}(M) \ge C(n)\cdot\Delta(M)$ for every PL $n$-manifold~$M\ne S^n$.
Note that $\Delta^{\text{st}}(S^n)=0$.
The (conjectural) inequality \eqref{equation:sum-lower-bound-conjecture} implies this conjecture, and if this conjecture holds for $M$, then $\Delta^{\text{st}}(M)>0$.
The following result gives evidence:

\begin{theoremalpha}
  \label{theorem:sum-stable-complexity-bound}
  Let $M$ be a closed orientable PL $(4k-1)$-manifold.
  Then
  \[
    \Delta^{\text{st}}(M) \ge C(4k-1) \cdot |\rho^{(2)}(M,\phi)|
  \]
  for all homomorphisms $\phi$ of $\pi_1(M)$, where $C(4k-1)>0$ is a constant that depends only on~$4k-1$.
  Consequently, $\Delta^{\text{st}}(M)>0$ if $\rho^{(2)}(M,\phi)\ne 0$ for some~$\phi$.
\end{theoremalpha}

For example, for the lens space $L_N=L(N;1,\ldots,1)$ of dimension $4k-1$ ($k\ge 1$) with $\pi_1(L_N)=\Z_N$, Theorem~\ref{theorem:sum-stable-complexity-bound} yields the following.
For more details, see Section~\ref{section:bounds-cheeger-gromov-rho}.

\begin{theorem-named}[Corollary to Theorem~\ref{theorem:sum-stable-complexity-bound}]
  For all $N>2$,
  \[
    C(4k-1)\cdot \Delta(L_N) \le \Delta^{\mathrm{st}}(L_N) \le \Delta(L_N)
  \]
  where $C(4k-1)>0$ is a constant depending only on $4k-1=\dim L_N$.
\end{theorem-named}

We finish the introduction with another application, which shows that
there exist manifolds with arbitrarily large complexity in a fixed (simple) homotopy type if the fundamental group has nontrivial torsion.

\begin{theoremalpha}
  \label{theorem:unbounded-complexity-fixed-homotopy-type}
  Let $M$ be a closed orientable PL (resp.\ smooth) $n$-manifold such that $\pi_1(M)$ has a nontrivial finite order element, $n=4k-1$, $k>1$.
  Then there is a closed orientable PL (resp.\ smooth) $n$-manifold $N$ which is simple homotopy equivalent to $M$ and has arbitrarily large complexity, detected by the $\rho$-invariant.
\end{theoremalpha}

Note that the fundamental group condition in Theorem~\ref{theorem:unbounded-complexity-fixed-homotopy-type} is necessary: if one allows trivial or torsion-free fundamental groups, the homotopy type of $M$ may contain only finitely many (PL or smooth) manifolds $N$ (e.g.\ when $M=S^n$, $T^n$ for $n\ge 5$) so that their complexity is bounded.

\subsubsection*{Organization of the paper}

In Section~\ref{section:simplicial-cellular-transversality}, we develop transversality methods for simplicial-cell complex target spaces, which give inverse image submanifolds and embedded cobordism, with quantitative analysis of the complexity.
In Section~\ref{section:BDH-preliminaries}, we discuss some properties of the Baumslag-Dyer-Heller acyclic groups and quantitative chain homotopy.
In Section~\ref{section:quantitative-pl-bordism}, we prove Theorem~\ref{theorem:main-quantitative-bordism}, using results in Sections~\ref{section:simplicial-cellular-transversality} and~\ref{section:BDH-preliminaries}\@.
In Section~\ref{section:quantitative-smooth-bordism}, we prove Theorem~\ref{theorem:main-quantitative-smooth-bordism}, which is a smooth version of Theorem~\ref{theorem:main-quantitative-bordism}\@.
In Section~\ref{section:bounds-cheeger-gromov-rho}, we discuss how our bordism result applies to obtain Theorems~\ref{theorem:main-rho-invariant}, \ref{theorem:sum-stable-complexity-bound} and~\ref{theorem:unbounded-complexity-fixed-homotopy-type}\@.
In Section~\ref{section:questions-for-n=1}, we discuss Questions~\ref{question:linear-bordism} and~\ref{question:linear-bordism-over-acyclic} for $n=1$.

\subsubsection*{Acknowledgements}

The authors thank F. Manin and S. Weinberger for helpful conversations, as well as anonymous referees for thoughtful comments.
The key technique in the proof of Theorem~\ref{theorem:main-quantitative-bordism} was discovered when the first named author visited the Max Planck Institute for Mathematics in Bonn in~2017.
He appreciates the hospitality and support of the MPIM\@.
This work was partially supported by the National Research Foundation of Korea grant 2019R1A3B2067839.
The second author was also partially supported by Horizon Europe ERC Grant number:101045750/Project acronym:HodgeGeoComb.

\section{Submanifolds and cobordism from simplicial-cellular transversality}
\label{section:simplicial-cellular-transversality}

To develop quantitative methods, we use triangulations of PL manifolds and view them as simplicial complexes.
Recall that for a triangulated PL manifold $M$, the \emph{complexity} $\Delta(M)$ is defined to be the number of simplices (of any dimension) in the triangulation.
Note that $\Delta(M)$ is equivalent to the number of top dimensional simplices up to constants depending only on the dimension:
if $\dim M=n$ and $M$ has $k$ $n$-simplices, then $k \le \Delta(M) \le 2^n \cdot k$.

In this paper, we also encounter non-manifolds which require a generalized notion of triangulation.
The key example is the geometric realization of a simplicial set.
For those, we use the notion of a \emph{simplicial-cell complex}, following~\cite{Cha:2014-1}.
Each cell of a simplicial-cell complex, which we call a simplex, is identified with a standard simplex.
A simplicial complex is a simplicial-cell complex.
We also use the notion of \emph{simplicial-cellular maps} between simplicial-cell complexes, which generalizes simplicial maps between simplicial complexes.
Refer to Section~\ref{subsection:simplicial-cellular-subdivision} for more details on simplicial-cell complexes and simplicial-cellular maps, including the definitions.

In this section, we show transversality results for a simplicial-cellular map $f\colon X\to P$ of a triangulated PL manifold $X$ to a simplicial-cell complex~$P$.
The main outcome is inverse image submanifolds and cobordisms between them with linearly bounded complexity.
It has somewhat distinct aspects when compared with the usual PL transversality approaches in the literature.
First, the usual transversality theorems perturb the given map $f\colon X\to P$ slightly to arrange it transverse to a given point (or submanifold) in~$P$, but perturbation often requires subdivision to make the map simplicial, and it potentially increases the complexity beyond the desired growth.
To avoid this, we show that $f$ is automatically transverse to interior points of every top-dimensional simplex of the target space~$P$, without perturbing $f$ or subdividing $X$ or~$P$.
Second, since the target space $P$ is not a simplicial complex in our case, we need to deal with additional technical sophistication, which was not considered in the literature.
We will describe relevant details in this section.

In what follows, $\hat\sigma$ denotes the barycenter of a simplex~$\sigma$.
For a simplex $\sigma=[v_0,\ldots,v_p]$, let $\sigma_0=[w_0,\ldots,w_p]$ be the simplex in the interior of $\sigma$ spanned by the vertices $w_i = \frac12(v_i+\hat\sigma)$.
We will often view $\sigma_0$ as a regular neighborhood of $\hat\sigma_0 = \hat\sigma$.

\begin{theorem}[Submanifold from simplicial-cellular transversality]
  \label{theorem:simplicial-cellular-transversality-submanifold}
  Suppose that $X$ is a simplicial complex with $\dim X=n$, $P$ is a simplicial-cell complex with $\dim P=p$ and $f\colon X\to P$ is a simplicial-cellular map.
  For each top-dimensional simplex $\sigma$ of $P$ (i.e.\ $\dim \sigma = p$), let $Y_\sigma = f^{-1}(\hat\sigma)$.

  \begin{enumerate}

    \item\label{item:simplicial-cellular-transversality}
    The subset $Y_\sigma$ is a subpolyhedron and has a regular neighborhood $N$ in $X$ such that $(N,Y_\sigma)$ is PL homeomorphic to $Y_\sigma \times (\sigma_0,\hat\sigma)$ and the restriction $f|_N$ is the projection $Y_\sigma \times \sigma_0 \to \sigma_0 \subset P$.

    \item\label{item:transversality-linear-complexity}
    Each $Y_\sigma$ has a triangulation with complexity $\Delta(Y_\sigma)$ such that
    \[
      \Delta\biggl(\bigcup_\sigma Y_\sigma\biggr) = \sum_{\sigma} \Delta(Y_\sigma) \le C(n) \cdot \Delta(X)
    \]
    where $\sigma$ varies over all $p$-simplices of $P$ and $C(n)$ is a constant depending only on the dimensions~$n$.

    \item\label{item:transversality-submanifold}
    If $X$ is a triangulated PL $n$-manifold, then $Y_\sigma$ is a locally flat PL $(n-p)$-submanifold properly embedded in~$X$.

  \end{enumerate}

\end{theorem}

The next theorem provides cobordism between the submanifolds $Y_\sigma$ with linearly bounded complexity.
To state it, we use the following convention and notation.
Suppose $f\colon X\to P$ is as in Theorem~\ref{theorem:simplicial-cellular-transversality-submanifold}.
Fix an orientation of each simplex of~$P$.
Orient $\sigma_0$ by the orientation of~$\sigma$.
When $X$ is oriented, orient the submanifold $Y_\sigma=f^{-1}(\hat\sigma)$ by comparing the product orientation of the product neighborhood $Y_\sigma\times \sigma_0$ in Theorem~\ref{theorem:simplicial-cellular-transversality-submanifold}~\ref{item:simplicial-cellular-transversality} with the orientation of~$X$.
Also, regard $Y_\sigma$ as framed by the product regular neighborhood $Y_\sigma \times \sigma_0$.
For a non-negative integer $k$, denote $k$ parallel copies of~$Y_\sigma$ taken along the framing by~$k Y_\sigma$.

Write the boundary of a $p$-simplex $\sigma$ of $P$ ($p=\dim P$) as $\partial\sigma = \sum_\tau d_{\tau\sigma} \tau$ in the cellular chain complex $C_*(P)$, where $\tau$ varies over $(p-1)$-simplices of $P$ and $d_{\tau\sigma}\in \Z$.

\begin{theorem}[Cobordism from simplicial-cellular transversality]
  \label{theorem:simplicial-cellular-transversality-cobordism}
  Suppose that $X$ is a triangulated PL $n$-manifold with empty boundary and $f\colon X\to P$ is a simplicial-cellular map to a simplicial-cell complex~$P$ with $p=\dim P$.

  \begin{enumerate}

    \item \label{item:inverse-image-cobordism}
    For each $(p-1)$-simplex $\tau$ of~$P$, there is a framed locally flat PL $(n-p+1)$-submanifold $Z_\tau$ in~$X$ which is bounded by the framed submanifold $\bigcup_\sigma k_{\tau\sigma} Y_\sigma$ where $k_{\tau\sigma}$ is a nonnegative integer such that $k_{\tau\sigma} \equiv d_{\tau\sigma} \bmod 2$ and $k_{\tau\sigma} \le p+1$.
    The $Z_\tau$ are pairwise disjoint.
    \item \label{item:inverse-image-oriented-cobordism}
    In addition, if $X$ is oriented (so that each $Y_\sigma$ is oriented as above), $Z_\tau$ can be oriented such that the boundary orientation is positive on $\frac{k_{\tau\sigma}+d_{\tau\sigma}}{2} Y_\sigma$ and negative on $\frac{k_{\tau\sigma}-d_{\tau\sigma}}{2} Y_\sigma$, i.e.,\ $\partial Z_\tau = \bigcup_\sigma d_{\tau\sigma} Y_\sigma$ algebraically.
    \item \label{item:cobordism-complexity}
    Each $Z_\tau$ has a triangulation such that $\Delta(\bigcup_\tau Z_\tau) = \sum_\tau \Delta(Z_\tau) \le C(n) \cdot \Delta(X)$ where $C=C(n)$ is a constant depending only on the dimension~$n$.

  \end{enumerate}

\end{theorem}

The remaining part of this section is devoted to the proof of Theorems~\ref{theorem:simplicial-cellular-transversality-submanifold} and~\ref{theorem:simplicial-cellular-transversality-cobordism}.
In Section~\ref{subsection:simplicial-cellular-subdivision}, we recall the definition of a simplicial-cell complex and prove a subdivision lemma.
In Sections~\ref{subsection:simplicial-cellular-transversality-submanifold} and~\ref{subsection:simplicial-cellular-transversality-cobordism}, we prove Theorems~\ref{theorem:simplicial-cellular-transversality-submanifold} and~\ref{theorem:simplicial-cellular-transversality-cobordism}, respectively.

\subsection{Simplicial-cell complexes and quantitative subdivision}
\label{subsection:simplicial-cellular-subdivision}

For the reader's convenience, we recall the definition of a simplicial-cell complex~\cite[Definition~3.6]{Cha:2014-1}.

\begin{definition}
  \label{definition:simplicial-cell-complex}
  A CW complex $X$ is a \emph{pre-simplicial-cell complex} if each $n$-cell is equipped with a characteristic map $(\Delta^n,\partial\Delta^n) \to (X^{(n)},X^{(n-1)})$, where $\Delta^n$ is the standard $n$-simplex, so that an open $n$-cell is identified with~$\inte\Delta^n$, the standard open simplex.
  We call an (open) $n$-cell of a pre-simplicial-cell complex an \emph{(open) $n$-simplex}.
  A map $f\colon X\to Y$ between pre-simplicial-cell complexes $X$ and $Y$ is \emph{simplicial-cellular} if $f$ sends each open simplex of $X$ onto an open simplex of $Y$, and the restriction of $f$ on an open simplex extends to a linear surjection between standard simplices $\Delta^k\to \Delta^\ell$ sending vertices to vertices.
  A pre-simplicial-cell complex is a \emph{simplicial-cell} complex if the attaching map $\partial\Delta^n \to X^{(n-1)}$ is simplicial-cellular for each simplex.
  That is, the restriction of $\partial\Delta^n \to X^{(n-1)}$ on the interior of a proper face of $\Delta^n$ is onto an open cell and induced by a linear surjection $\Delta^k\to \Delta^\ell$ sending vertices to vertices.
\end{definition}

For a simplex $\sigma$ of a simplicial-cell complex $X$, we denote the characteristic map by $\phi_\sigma\colon \Delta^n \to X$.

The geometric realization of a simplicial set $X$ is a simplicial-cell complex whose simplices are in 1-1 correspondence with non-degenerated simplices of~$X$~\cite{Milnor:1957-3} (see also~\cite[Appendix~A]{Cha:2014-1} for more details).
So, when we view a simplicial set $X$ as a simplicial-cell complex, a simplex means a non-degenerated simplex of~$X$.

\begin{remark}
  The geometric realization of a simplicial set has an additional property:
  for each proper face of an $n$-simplex $\sigma$, the restriction of the attaching map $\partial\Delta^n \to X^{(n-1)}$ on the face is induced by a linear surjection $\Delta^k\to \Delta^\ell$ which preserves the \emph{order} of vertices.
  (Hatcher calls such a complex a singular $\Delta$-complex~\cite{Hatcher:2002-1};
  one may view a simplicial-cell complex as an unordered version, which one could call an ``unordered singular $\Delta$-complex.'')
  In this paper, our arguments do not require the orderedness of vertices.
\end{remark}

Let $L$ be a simplicial-cell complex with $\dim L = p$.
We will use the following form of subdivision, which subdivides top-dimensional simplices only.
For each $p$-simplex $\sigma$ of~$L$, remove it from $L$, choose a (simplicial complex) subdivision $L'_\sigma$ of the simplicial complex $\Delta^p$ which does not subdivide faces of $\Delta^p$ (i.e.,\ all new vertices are in $\inte\Delta^p$) and attach $L'_\sigma$ to the $(p-1)$-skeleton $L^{(p-1)}$ along the attaching map $\partial\Delta^p \to L^{(p-1)}$ of $\sigma$.
(Here $L'_\sigma$ may be $\Delta^p$ itself, i.e., we allow unsubdivided $p$-simplices.)
Let $L'$ be the resulting complex.
Since $\partial\Delta^p$ is not subdivided in $L'_\sigma$, $L'$ is a simplicial-cell complex.
We call $L'$ a \emph{top-dimensional subdivision} of~$L$.

\begin{lemma}
  \label{lemma:simplicial-cellular-subdivision}
  Let $f\colon K\to L$ be a simplicial-cellular map of a simplicial complex $K$ to a simplicial-cell complex $L$ with $k=\dim K$, $p=\dim L$, and let $L'$ be a top-dimensional subdivision of~$L$.
  Suppose that $K$ has $n$ simplices and every $p$-simplex of $L$ is subdivided into at most $d$ simplices in~$L'$.
  Then there is a (simplicial complex) subdivision $K'$ of $K$ such that $f\colon K'\to L'$ is simplicial-cellular and $K'$ has at most $C \cdot n$ simplices, where $C = C(k,p,d)$ is a constant depending only on $k$, $p$ and~$d$.
\end{lemma}

\begin{proof}
  In our $K'$, simplices of $K$ that $f$ sends to $L^{(p-1)}$ will be left unsubdivided.
  Those will be denoted by~$A$ in this proof.
  Denote by $B$ a simplex of $K$ such that $f(B)=\sigma$ is a $p$-simplex $\sigma$ of~$L$, i.e.,
  $f|_B = \phi_\sigma \circ \eta_B$ for some linear surjection $\eta_B\colon B \to \Delta^p$ sending vertices to vertices.
  (Recall that $\phi_\sigma\colon \Delta^p\to L$ is the characteristic map.)
  We will subdivide $B$ by generalizing an argument in~\cite[Lemma~1.8]{Hudson:1969-1} (see also \cite[Lemma~2.16]{Rourke-Sanderson:1972-1}).
  Recall that a subdivision $L'_\sigma$ of $\Delta^p$ is used to subdivide~$\sigma$.
  For each simplex $\tau$ of~$L'_\sigma$, $\eta_B^{-1}(\tau)$ is a linear convex cell in~$B$ (i.e.,\ convex hull of finitely many points).
  We claim that the collection of cells
  \[
    \cK = \{A \mid f(A)\subset L^{(p-1)} \} \cup
    \{\eta_B^{-1}(\tau) \mid \text{$f(B)=\sigma$, $\dim\sigma=p$, $\tau$ is a simplex of $L'_\sigma$} \}
  \]
  is a cell complex structure on (the underlying space of)~$K$ in the sense of~\cite[p.~5]{Hudson:1969-1}, \cite[p.~14]{Rourke-Sanderson:1972-1}.
  It means that
  \begin{enumerate}[label=(\roman*)]
    \item\label{item:cell-complex-face}
    faces of a cell in~$\cK$ are in $\cK$, and
    \item\label{item:cell-complex-intersection}
    the intersection of two cells in $\cK$ is in~$\cK$.
  \end{enumerate}
  First, we will verify~\ref{item:cell-complex-face}.
  If $f(A)\subset L^{(p-1)}$, any face $A_1$ of $A$ is sent into $L^{(p-1)}$ too, so $A_1\in\cK$.
  A face of $\eta_B^{-1}(\tau)$ is of the form $B_1\cap \eta_B^{-1}(\tau_1)$ where $B_1$ and $\tau_1$ are (not necessarily proper) faces of $B$ and~$\tau$.
  If $\eta_B(B_1)=\Delta^p$, then $\eta_{B_1} = \eta_B|_{B_1}$ and thus $B_1\cap \eta_B^{-1}(\tau_1) =\eta_{B_1}^{-1}(\tau_1)$ is in~$\cK$.
  If $\eta_B(B_1)$ is a proper face of $\Delta^p$, then since the subdivision $L_\sigma'$ of $\sigma$ does not subdivide proper faces of $\Delta^p$, $\eta_B(B_1)\cap \tau_1$ is a face of the simplex $\eta_B(B_1)$ (or empty).
  It follows that $B_1\cap \eta_B^{-1}(\tau_1)$ is a face of $B_1$ (or empty).
  In particular, $B_1\cap \eta_B^{-1}(\tau_1)$ is a simplex of~$K$.
  Since $\eta_B(B_1) \subset \partial\Delta^p$, $f(B_1\cap \eta_B^{-1}(\tau_1)) \subset f(B_1)\subset L^{(p-1)}$, and thus the simplex $B_1\cap \eta_B^{-1}(\tau_1)$ is in~$\cK$.
  This shows~\ref{item:cell-complex-face}.

  To verify~\ref{item:cell-complex-intersection}, consider the following three types of intersections:
  \begin{enumerate}
    \item If $A_1\cap A_2 \neq \emptyset$, then since $f(A_1\cap A_2) \subset f(A_1) \subset L^{(p-1)}$, $A_1\cap A_2$ is a cell in~$\cK$.
    \item Suppose $A\cap \eta_B^{-1}(\tau) \ne \emptyset$.
    Since $f(A\cap B) \subset f(A)\subset L^{(p-1)}$, we have $\eta_B(A\cap B)\subset \partial\Delta^p$.
    Since the subdivision $L_\sigma'$ of $\Delta^p$ does not subdivide simplices in $\partial\Delta^p$, $\tau\cap \partial\Delta^p$ is a proper face of $\Delta^p$ and $A\cap \eta_B^{-1}(\tau) = (\eta_B|_{A\cap B})^{-1} (\tau\cap\partial\Delta^p)$ is a face $A_1$ of $A\cap B$.
    Since $f(A_1) \subset f(A) \subset L^{(p-1)}$, $A_1$ is a cell of~$\cK$.
    \item Suppose that $\eta_{B_1}^{-1}(\tau_1)\cap \eta_{B_2}^{-1}(\tau_2) \neq \emptyset$.

    If $f(B_1\cap B_2) \subset L^{(p-1)}$, then each of $\eta_{B_1}$ and $\eta_{B_2}$ sends $B_1\cap B_2$ to $\partial \Delta^p$.
    Since the subdivision $L_\sigma'$ of $\Delta^p$ does not subdivide simplices in $\partial\Delta^p$, it follows that $\eta_{B_1}^{-1}(\tau_1) \cap B_2 = \eta_{B_1}^{-1}(\tau_1')$ for some proper face $\tau_1'$ of $\Delta^p$, and thus $\eta_{B_1}^{-1}(\tau_1) \cap B_2$ is a face of $B_1\cap B_2$.
    Similarly, $\eta_{B_2}^{-1}(\tau_2) \cap B_1$ is a face of $B_1\cap B_2$.
    Thus $\eta_{B_1}^{-1}(\tau_1) \cap \eta_{B_2}^{-1}(\tau_2)$ is a face $A$ of the simplex $B_1 \cap B_2$.
    Since $f(A)\subset f(B_1 \cap B_2) \subset L^{(p-1)}$, $A$ is a cell of~$\cK$.

    If $f(B_1\cap B_2)$ is not contained in $L^{(p-1)}$, then $\eta_{B_1}(B_1\cap B_2) = \Delta^p = \eta_{B_2}(B_1\cap B_2)$.
    It follows that $\eta_{B_1} = \eta_{B_2}$ on $B_1\cap B_2$ and $\eta_{B_1\cap B_2}$ is the restriction of $\eta_{B_1}$ (and $\eta_{B_2}$) on $B_1\cap B_2$.
    Since $\eta_{B_1}^{-1}(\tau_1) \cap \eta_{B_2}^{-1}(\tau_2) \subset B_1\cap B_2$, it follows that $\eta_{B_1}^{-1}(\tau_1) \cap \eta_{B_2}^{-1}(\tau_2) = \eta_{B_1\cap B_2}^{-1}(\tau_1\cap \tau_2)$.
    This is a cell of~$\cK$.
  \end{enumerate}
  This shows~\ref{item:cell-complex-intersection}.

  Now, apply the ``pulling triangulation'' \cite[Lemma~1.4]{Hudson:1969-1} (see also \cite[Lemma~2.9]{Rourke-Sanderson:1972-1}) to subdivide $\cK$ into a simplicial complex~$K'$ without introducing new vertices.
  Since we need to use this triangulation again later, we state it as a lemma below.

  \begin{lemma}[{Pulling triangulation~\cite[Lemma~1.4]{Hudson:1969-1}}]
    \label{lemma:pulling-triangulation}
    A cell complex $\cK$ can be subdivided into a simplicial complex without adding extra vertices.
  \end{lemma}

  \begin{proof}
    Choose an order of vertices of~$\cK$.
    For each cell $E$ of $\cK$, write $E$ as a join $E=v_0\cdot F$ where $v_0$ is the first vertex of $E$ and let $F$ be the union of all faces of $E$ not containing~$v_0$, and triangulate $E$ by taking the join of $v_0$ and the triangulation of $F$ given by the induction hypothesis.
    Note that if $B$ is a face of $E$ containing $v_0$, then $v_0$ is the first vertex of~$B$, so that the join triangulation of $E=v_0\cdot F$ agrees with the inductively given triangulation on~$B$.
    Thus the triangulation is well-defined.
  \end{proof}

  Returning to the proof of Lemma~\ref{lemma:simplicial-cellular-subdivision}, note that $f$ sends each cell $E$ of $\cK$ linearly to a simplex of $L'$, i.e., there is a simplex $\sigma'$ of $L'$ and a linear map $\eta_E\colon E\to \Delta^q$ ($q=\dim\sigma'$), sending vertices to vertices, such that $f|_E = \phi_{\sigma'} \circ \eta_E$.
  Since the subdivision $K'$ of $\cK$ does not introduce any extra vertex, $\eta_E$ sends a simplex of $K'$ contained in $E$ onto a face of $\Delta^q$ linearly.
  It follows that $f\colon K' \to L'$ is simplicial-cellular.

  To complete the proof, estimate the number of simplices in~$K'$ as follows.
  Recall that a simplex $B$ of $K$ is subdivided in $K'$ only if $f(B)=\sigma$ is a $p$-simplex of~$L$.
  For each simplex $\tau$ of the subdivision $L_\sigma'$ of $\Delta^p$, the cell $E=\eta_B^{-1}(\tau)$ is defined by at most $k + p + 2$ linear inequalities, since $B$ and $\tau$ are defined by at most $k+1$ and $p+1$ inequalities respectively.
  Since a vertex of $E$ is determined by a subcollection of those linear inequalities (by solving the system of associated linear equalities), it follows that $E$ has at most $2^{k+p+2}$ vertices.
  Since no extra vertices are added in the subdivision of $E$, a simplex in the subdivision is determined by choosing a subset of the set of vertices of~$E$.
  So the subdivision of $E$ has at most $2^{2^{k+p+2}}$ simplices.
  Since there are at most $d$ simplices $\tau$ in the subdivision $L_\sigma'$ of $\Delta^p$, the simplex $B$ of $K$ is subdivided into at most $C = d\cdot 2^{2^{k+p+2}}$ simplices in~$K'$.
  Since $K$ has $n$ simplices, it follows that $K'$ has at most $C \cdot n$ simplices.
\end{proof}

\subsection{Submanifolds from simplicial-cellular transversality}
\label{subsection:simplicial-cellular-transversality-submanifold}

Recall that the barycenter of a $p$-simplex $\sigma$ in a simplicial-cell complex $P$ is defined to be $\hat\sigma = \phi_\sigma\bigl(\frac1{p+1} (e_0 + \cdots + e_p)\bigr) \in \inte\sigma\subset P$, where $\phi_\sigma\colon \Delta^p \to P$ is the characteristic map and $e_0,\ldots,e_p$ are the vertices of~$\Delta^p$.

\begin{lemma}
  \label{lemma:basic-facts-barycenter-inverse}
  Let $X$ be a simplicial complex and $P$ be a simplicial-cell complex.
  Let $f\colon X\to P$ be a simplicial-cellular map.
  Let $A$ and $\sigma$ be simplices of $X$ and~$P$, respectively.
  Let $Y_\sigma = f^{-1}(\hat\sigma) \subset X$.
  Then the following hold.
  \begin{enumerate}
    \item $Y_\sigma \cap \inte A \ne \emptyset$ if and only if $f(A)=\sigma$.
    \item If $\sigma$ is top-dimensional, i.e.,\ $\dim\sigma = p$, then $Y_\sigma \cap A \ne \emptyset$ if and only if $f(A)=\sigma$.
    Consequently, by (1), $Y_\sigma \cap \inte A \ne \emptyset$ if and only if $Y_\sigma \cap A \ne \emptyset$.
    \item If $f(A)=\sigma$, then $Y_\sigma \cap A$ is a $(k-p)$-disk which is unknotted (in fact linearly embedded) in~$A$.
  \end{enumerate}
\end{lemma}

\begin{proof}\leavevmode
  \Nopagebreak
  \begin{enumerate}
    \item Recall that $f$ sends $A$ to a simplex $\tau$ of $P$ and $f(\inte A) = \inte\tau$.
    Since distinct simplices of $P$ have disjoint interiors and $\hat\sigma$ lies in $\inte \sigma$, the conclusion follows.

    \item If $\hat\sigma\in f(A)$, then since $A=\bigcup \{\inte A_1 \mid A_1$ is a face of $A\}$, $\hat\sigma\in f(\inte A_1)$ for some face $A_1$ of~$A$.
    So, $f(A_1) = \sigma$ by~(1).
    Also, $f(A_1) = \sigma$ is a face of the simplex~$f(A)$ since $A_1$ is a face of~$A$.
    It follows that $f(A)=\sigma$ since $\sigma$ is top-dimensional.

    \item We use join coordinates to describe $\eta_A\colon A\to \Delta^p$ and $Y_\sigma\cap A$ as follows.
    Write $\Delta^p=[e_0,\ldots,e_p]$.
    Let $A_j=\eta_A^{-1}(e_j)$ be the face of $A$ spanned by the vertices that $\eta_A$ sends to~$e_j$.
    Then $A = A_0 \cdots A_p$, the join of $A_0,\ldots,A_p$.
    An element $x\in A$ is written in the form $x = \sum_j \lambda_j x_j$ where $x_j\in A_j$, $\sum_j \lambda_j = 1$, $\lambda_j\ge 0$, and $\eta_A\bigl( \sum_j \lambda_j x_j \bigr) = \sum_j \lambda_j e_j$ on~$A$.
    Therefore
    \[
      Y_\sigma \cap A = \eta_A^{-1} \biggl(\frac1{p+1} \sum_j e_j\biggr) = \frac1{p+1} \sum_j \eta_A^{-1}(e_j) = \sum_j \frac1{p+1} A_j
    \]
    in $A = A_0\cdots A_p$, and thus $Y_\sigma \cap A \cong A_0\times\cdots\times A_p \cong D^{k-p}$.
    Also, from this description of $Y_\sigma \cap A \subset A$, it follows that the convex linear cell $A_0\times\cdots\times A_p$ embeds linearly onto $Y_\sigma \cap A$, and thus $Y_\sigma \cap A$ is unknotted.
    \qedhere
  \end{enumerate}
\end{proof}

Now we begin the proof of Theorem~\ref{theorem:simplicial-cellular-transversality-submanifold}.
Let $f\colon X\to P$ be a simplicial-cellular map where $X$ is a simplicial complex with $\dim X = n$ and $P$ is a simplicial-cell complex with $\dim P = p$.
Note that if $p>n$, then $Y_\sigma = f^{-1}(\hat\sigma)= \emptyset$ and so Theorem~\ref{theorem:simplicial-cellular-transversality-submanifold} holds vacuously.
Thus we may assume that $p\le n$.

\begin{proof}
  [Proof of Theorem~\ref{theorem:simplicial-cellular-transversality-submanifold}~\ref{item:simplicial-cellular-transversality}]

  Let $\sigma$ be a $p$-simplex in $P$, $p=\dim P$.
  Let $D\subset \inte\sigma$ be a $p$-disk neighborhood of the barycenter~$\hat\sigma$.
  Our goal is to show the following: $Y_\sigma=f^{-1}(\hat\sigma)$ is a subpolyhedron in $X$ which has a regular neighborhood $N$ such that $(N,Y_\sigma)$ is PL homeomorphic to $Y_\sigma \times (D,\hat\sigma)$ and the restriction $f|_N\colon N\to P$ is equal to the projection $Y_\sigma \times D \to D \subset P$.
  (When $D$ is the $p$-simplex $\sigma_0 \subset \inte\sigma$, one obtains the statement of Theorem~\ref{theorem:simplicial-cellular-transversality-submanifold}.)

  Let $A$ be a $k$-simplex of~$X$.
  For brevity, denote $Y_\sigma \cap A$ by~$E_A$.
  Note that $E_A\neq \emptyset$ if and only if $f(A)=\sigma$ by Lemma~\ref{lemma:basic-facts-barycenter-inverse}.
  When $f(A)=\sigma$, let $\eta_A\colon A\to \Delta^p=[e_0,\ldots,e_p]$ be the linear surjection such that $f=\phi_\sigma \circ \eta_A$ on $A$ and let $A_j=\eta_A^{-1}(e_j)$ as before, so that $\eta_A$ sends an element $x = \sum_j \lambda_j x_j \in A$ to $\eta_A(x) = \sum_j \lambda_j e_j$.
  Let
  \[
    N_A = (\phi_\sigma\circ\eta_A)^{-1}(D) = \biggl\{ \sum_j \lambda_j x_j \in A \;\bigg|\; x_j\in A_j, \, \phi_\sigma\biggl(\sum_j \lambda_j e_j\biggr) \in D\subset \inte\sigma \biggr\},
  \]
  a neighborhood of $E_A$ in~$A$.
  Define a map $h_A\colon E_A \times D \to N_A$ by
  \[
    h_A \Biggl( \sum_j \frac1{p+1} x_j ,\, \phi_\sigma\bigg(\sum_j \lambda_j e_j \biggr) \Biggr) = \sum_j \lambda_j x_j.
  \]
  It is readily verified that $h_A$ is a well-defined PL homeomorphism, using that $\lambda_j\ne 0$ if $\phi_\sigma\bigl(\sum_j \lambda_j e_j \bigr) \in D\subset \inte\sigma$.
  In addition, for each face $A'$ of $A$ that meets~$Y_\sigma$,
  $h_{A'}$ is the restriction of~$h_A$ on $E_{A'}\times D \subset E_A\times D$.
  It follows that the cells $\{E_A\}$ define a cell complex structure of~$Y_\sigma$ and the local product structure $\{N_A \cong E_A\times D\}$ gives rise to a well-defined trivial normal $p$-block bundle structure on~$N = \bigcup N_A$ over~$Y_\sigma$, so that $(N,Y_\sigma) \cong Y_\sigma\times(D,\hat\sigma) \cong Y_\sigma\times(D^p,0)$.
\end{proof}

\begin{proof}[Proof of Theorem~\ref{theorem:simplicial-cellular-transversality-submanifold}~\ref{item:transversality-linear-complexity}]
  Triangulate the disjoint union $\bigcup_\sigma Y_\sigma$ by subdividing the cell complex structure $\{E_A\}$ without adding extra vertices, by using the pulling triangulation in Lemma~\ref{lemma:pulling-triangulation}.

  The cell complex $\{E_A\}$ has at most $\Delta(X)$ cells when $X$ has $\Delta(X)$ simplices.
  Recall from the proof of Lemma~\ref{lemma:basic-facts-barycenter-inverse} that $E_A$ is (empty or) linearly isomorphic to the product of $p+1$ simplices $A_0,\ldots,A_p$ with $\dim A_j \le \dim A \le n$.
  So the number of vertices of the cell $E_A$ is at most $(n+1)^{p+1} \le (n+1)^{n+1}$.
  When the cell $E_A$ has $s$ vertices, the pulling triangulation of $E_A$ has at most $2^s$ simplices.
  Thus, the triangulation of $\bigcup_\sigma Y_\sigma$ has at most $C \cdot \Delta(X)$ simplices, where $C = 2^{((n+1)^{(n+1)})}$ is a constant depending only on~$n$.
\end{proof}

\begin{proof}[Proof of Theorem~\ref{theorem:simplicial-cellular-transversality-submanifold}~\ref{item:transversality-submanifold}]

  Suppose $X$ is a PL $n$-manifold equipped with a combinatorial triangulation.
  Our goal is to show that $Y_\sigma$ is a properly embedded PL $(n-p)$-submanifold in~$X$.
  Let $A$ be a $k$-simplex in~$X$ such that $E_A \ne \emptyset$.

  \begin{assertion}
    In $\st(A)$, $Y_\sigma \cap \st(A)$ is isotopic to the join $E_A \cdot \lk(A)$, rel $E_A$.
  \end{assertion}

  See Figure~\ref{figure:inverse-image-local} for an illustrated example of the assertion.

  \begin{figure}[ht]
    \begin{tikzpicture}[
      x=1cm,y=1cm,scale=1.3,>=stealth,
      md/.style={line width=1pt,line cap=butt},
      mdover/.style={line width=2pt,color=white,double=black,
        double distance=1pt,line cap=butt},
      thin/.style={md,line width=.5pt},
      thinover/.style={mdover,line width=1.5pt,double distance=.5pt},
      md,
      ]
      \def\c#1{coordinate(#1) node[text=blue]{\tiny#1}} % for debug
      \def\c#1{coordinate(#1)} % for production
      \small
  
      \def\r{2 and 1.1}
      \def\drawstar{
        \fill (0,0) \c{O}
        (0,-2) \c{S} circle(1.2pt) (0,2) \c{N} circle(1.2pt)
        (15:\r) \c{A} circle(1.2pt) (100:\r) \c{B} circle(1.2pt)
        (165:\r) \c{C} circle(1.2pt) (205:\r) \c{D} circle(1.2pt)
        (-35:\r) \c{E} circle(1.2pt);
    
        \draw (N)--(S)
        (S)--(A) (S)--(B) (S)--(C) (S)--(D) (S)--(E)
        (A)--(B)--(C)--(D)--(E)--cycle
        (N)--(A) (N)--(B) (N)--(C) (N)--(D) (N)--(E);
        \draw[mdover] ($(N)!0.2!(S)$) -- ($(N)!0.8!(S)$)
        ($(N)!0.2!(D)$) -- ($(N)!0.95!(D)$)
        ($(N)!0.2!(E)$) -- ($(N)!0.95!(E)$);
        \draw[mdover] ($(D)!0.05!(E)$) -- ($(D)!0.95!(E)$);
      }
      \drawstar
      \draw (O) node[right]{$E_A$}
      (S) node[below]{$v_0$} (N) node[above]{$v_1$}
      (A) node[right]{$v_2$} (B) node[shift={(-0.2,-0.35)}]{$v_3$}
      (C) node[left]{$v_4$} (D) node[left]{$v_5$} (E) node[right]{$v_6$};

      \begin{scope}[thin,fill=blue,color=blue,fill opacity=0.2]
        \filldraw (O) -- ($(N)!0.5!(A)$) -- ($(N)!0.5!(B)$) -- cycle;
        \filldraw (O) -- ($(N)!0.5!(B)$) -- ($(B)!0.5!(C)$) -- ($(C)!0.5!(S)$) -- cycle;
        \filldraw (O) -- ($(C)!0.5!(S)$) -- ($(D)!0.5!(S)$) -- cycle;
        \filldraw (O) -- ($(D)!0.5!(S)$) -- ($(E)!0.5!(S)$) -- cycle;
        \filldraw (O) -- ($(E)!0.5!(S)$) -- ($(E)!0.5!(A)$) -- ($(N)!0.5!(A)$) -- cycle;
      \end{scope}
      
      \fill[color=blue] (O) circle(1.2pt) ($(N)!0.5!(A)$) circle(1.2pt)
      ($(N)!0.5!(B)$) circle(1.2pt) ($(B)!0.5!(C)$) circle(1.2pt) ($(C)!0.5!(S)$) circle(1.2pt) 
      ($(D)!0.5!(S)$) circle(1.2pt) ($(E)!0.5!(S)$) circle(1.2pt) ($(A)!0.5!(E)$) circle(1.2pt);

      \draw (-3,0) node[scale=1.2]{$\st(A) = $};
      \draw[->] (2.5,0) -- (4.5,0) node[midway,above]{$f$}
      node[midway,below]{
        $\begin{aligned}
          v_0, v_2, v_3 & \mapsto e_0 \\[-.75ex]
          v_1, v_4, v_5, v_6 & \mapsto e_1
        \end{aligned}$
      };
  
      \begin{scope}[shift={(5,0)}]
        \filldraw (0,-1) circle(1pt) node[right]{$e_0$} -- (0,1) circle(1pt) node[right]{$e_1$};
        \fill[color=blue] circle(1.2pt) node[right]{$\hat\sigma$};
      \end{scope}

      \draw (2.5/2,-3.5/2) node[scale=2,rotate=-45]{$\approx$};
      \begin{scope}[shift={(2.5,-3.5)}]
        \drawstar
        \begin{scope}[thin,color=red,fill opacity=0.2]
          \fill (O) -- (A) -- (B) -- cycle;
          \fill (O) -- (B) -- ($(B)!0.5!(C)$) -- (C) -- cycle;
          \fill (O) -- (C) -- (D) -- cycle;
          \fill (O) -- (D) -- (E) -- cycle;
          \fill (O) -- (E) -- ($(E)!0.5!(A)$) -- (A) -- cycle;
          \foreach \v in {A,B,C,D,E} { \draw (O)--(\v); }
        \end{scope}
        \fill[color=red] (O) circle(1.2pt)
        (A) circle(1.2pt) (B) circle(1.2pt) (C) circle(1.2pt) (D) circle(1.2pt) (E) circle(1.2pt) 
        ($(B)!0.5!(C)$) circle(1.2pt) ($(A)!0.5!(E)$) circle(1.2pt);
      \end{scope}
    \end{tikzpicture}

    \caption{
      An illustration of the assertion.
      Here $n=3$, $p=1$, $\sigma=[e_0,e_1]$, $A=[v_0,v_1]$, and $f$ is a simplicial map sending vertices as indicated.
      The inverse image $Y_\sigma \cap \st(A) = f^{-1}(\hat\sigma) \cap \st(A)$ is the blue-shaded 2-complex, and $E_A = Y_\sigma \cap A$ is a single point in this example.
      In $\st(A)$, $Y_\sigma \cap \st(A)$ is isotopic rel $E_A$ to the join $E_A\cdot \lk(A)$, which is the red-shaded 2-complex.
    }
    \label{figure:inverse-image-local}
  \end{figure}
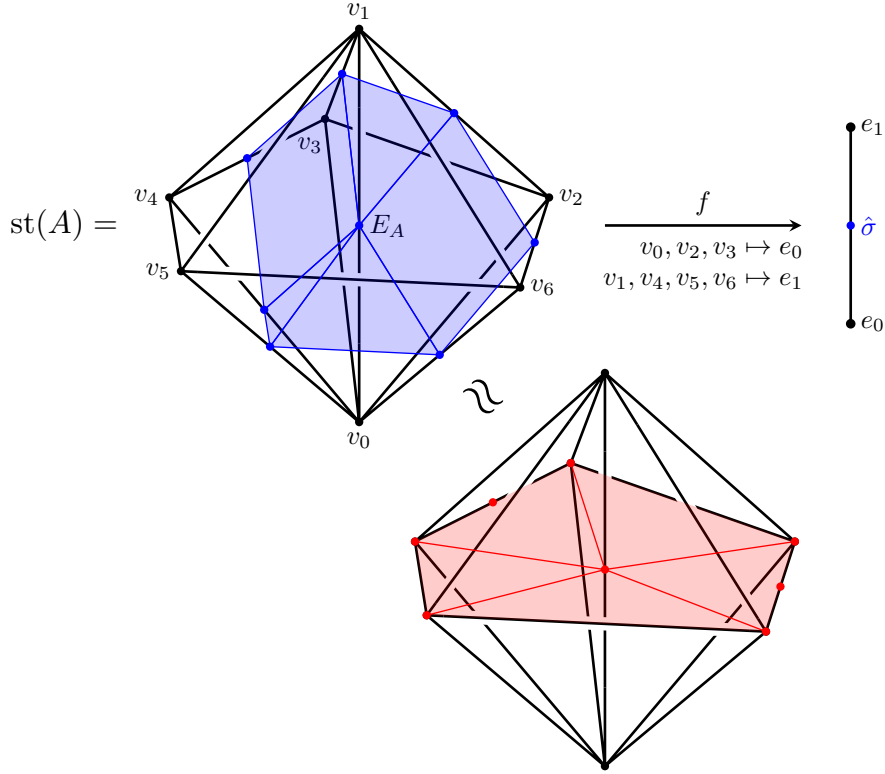

  Since $X$ is an $n$-manifold, $\lk(A)$ is an $(n-k-1)$-disk if $A\subset \partial X$, an $(n-k-1)$-sphere otherwise.
  Note that $(\st(A), E_A\cdot lk(A)) = (A, E_A) \cdot \lk(A)$ and $(A,E_A)\cong (D^k, D^{k-p})$ is an unknotted disk pair by Lemma~\ref{lemma:basic-facts-barycenter-inverse}.
  So, from the above assertion, it follows that $(\st(A), Y_\sigma \cap \st(A)) \cong (\st(A), E_A\cdot\lk(A))\cong (D^n, D^{n-p})$ is an unknotted disk pair.
  Thus $Y_\sigma$ is a locally flat $(n-p)$-submanifold in~$X$.
  Also, note that $E_A = Y_\sigma \cap A$ lies in the boundary of $Y_\sigma \cap \st(A)\cong E_A\cdot\lk(A) \cong  D^{n-p}$ if and only if $\lk(A)$ is an $(n-k-1)$-disk, i.e., exactly when $E_A \subset A\subset \partial X$.
  Therefore $Y_\sigma$ is properly embedded in~$X$.
  This proves the conclusion.

  The assertion follows from the following statement, which we will prove by induction on~$r$: for each $r$-simplex $B\subset \lk(A)$, there is an isotopy $h_B\colon E_{A\cdot B} \approx E_A\cdot B$ from the disk $E_{A\cdot B}$ to $E_A\cdot B$ (through properly embedded disks) rel $A$ in $A\cdot B$ such that for each face $B'$ of~$B$, $h_B$ restricts to the isotopy~$h_{B'}\colon E_{A\cdot B'} \approx E_A \cdot B'$ rel $A$ in~$A\cdot B$.

  By Lemma~\ref{lemma:basic-facts-barycenter-inverse}, $E_{A\cdot B}$ is an unknotted disk in the simplex $A\cdot B\subset X$.
  Also, since $E_A$ is unknotted in~$A$, the disk $E_{A} \cdot B$ is unknotted in $A\cdot B$.
  Since $E_{A\cdot B} \cap A = E_A = (E_{A} \cdot B) \cap A$, the unknotted disks $E_{A\cdot B}$ and $E_A \cdot B$ are isotopic rel~$A$ in $A\cdot B$.
  If $r=0$, then since $B$ has no proper face, we can take this isotopy as the desired~$h_B$.

  If $r\ge 1$, then by the induction hypothesis for $r-1$, we already have isotopies $h_{B'}$ in $A\cdot B'$ for proper faces $B'$ of $B$, whose union is an isotopy of $Y_\sigma\cap (A\cdot \partial B)$ to $E_A \cdot \partial B$ rel $A$ in the ball $A\cdot \partial B$.
  Denote this isotopy by~$h_{\partial B}\colon Y_\sigma\cap (A\cdot \partial B) \approx E_A \cdot \partial B$.
  We will define a desired isotopy $h_A \colon E_{A\cdot B} \approx E_A \cdot B$ rel $A$ in $A\cdot B$ that extends~$h_{\partial B}$ as follows.
  Note that $(A\cdot B, A\cdot \partial B, A) \cong (D^{k+r}\times I, D^{k+r}\times 0, D^k\times 0)$ is a standard triple.
  The isotopy $h_{\partial B}$ in $A\cdot \partial B$ is covered by an ambient isotopy $\{\phi_t \colon A\cdot B \to A\cdot B\}_{0 \le t \le 1}$ rel~$A$, $\phi_0=\id$.
  We have that $D = \phi_1(E_{A\cdot B})$ is an unknotted disk which intersects $A\cdot \partial B$ at $E_A \cdot \partial B$.
  Since $E_A \cdot B$ is another unknotted disk in $A\cdot B$ having the same intersection with $A\cdot \partial B$, there is an ambient isotopy $\{\psi_t \colon A\cdot B \to A\cdot B\}_{0 \le t \le 1}$ rel $A\cdot \partial B$ such that $\psi_0=\id$, $\psi_1(D)=E_A\cdot B$.
  The composition $\{\psi_t \circ \phi_t\}$ induces a desired isotopy $h_A\colon E_{A\cdot B} \approx \psi_1(\phi_1(E_{A\cdot B}))=E_A\cdot B$ rel $A$, which extends~$h_{\partial B}$.
\end{proof}

\subsection{Cobordism from simplicial-cellular transversality}
\label{subsection:simplicial-cellular-transversality-cobordism}

\begin{proof}
  [Proof of Theorem~\ref{theorem:simplicial-cellular-transversality-cobordism}~\ref{item:inverse-image-cobordism} and~\ref{item:inverse-image-oriented-cobordism}]

  Let $X$ be a triangulated PL $n$-manifold with $\partial X = \emptyset$ and $f\colon X\to P$ be a simplicial-cellular map to a simplicial-cell complex~$P$ with $p=\dim P$.
  Our goal is to construct, for each $(p-1)$-simplex $\tau$ of $P$, a framed $(n-p+1)$-submanifold $Z_\tau$ bounded by certain parallel copies of the framed submanifolds $Y_\sigma$ in~$X$.

  Similarly to the proof of Theorem~\ref{theorem:simplicial-cellular-transversality-submanifold}, we may assume that $p\le n$, since every $Y_\sigma$ is empty otherwise.

  We begin by subdividing $X$ and $P$ in such a way that the exteriors of $Y_\sigma\subset X$ and $\hat\sigma\in \sigma\subset P$ are subcomplexes, as detailed below.

  Write the standard $p$-simplex as $\Delta^p=[e_0,\ldots,e_p]$.
  Let $b$ be the barycenter of~$\Delta^p$.
  Let $v_j = \frac12(b + e_j)$ so that the $p$-simplex $[v_0,\ldots,v_p]$ is a regular neighborhood of~$b$.
  See Figure~\ref{figure:top-dim-subdivision}.
  The exterior $\Delta^p \sm \inte[v_0,\ldots,v_p]$ is a cell complex with $p+1$ $p$-cells, which are spanned by $\{e_0,\ldots,e_p,v_0,\ldots,v_p\} \sm \{e_j,v_j\}$, $j=0,\ldots,p$.
  Triangulate $\Delta^p \sm \inte[v_0,\ldots,v_p]$ without adding extra vertices, using the pulling triangulation in Lemma~\ref{lemma:pulling-triangulation}.
  Adjoining $[v_0,\ldots,v_p]$, we obtain a simplicial subdivision of~$\Delta^p$.
  Subdivide each $p$-simplex $\sigma$ of $P$ using this subdivision of $\Delta^p$, to define a top-dimensional subdivision $P'$ of the simplicial-cell complex~$P$.
  Let $\sigma_0 = \phi_\sigma[v_0,\ldots,v_p]$.
  The $p$-simplex $\sigma_0$ is a neighborhood of the barycenter~$\hat\sigma$.
  By the construction, $P'$ restricts to a simplicial-cell complex structure on $P\sm \bigcup_\sigma \inte\sigma_0$.

  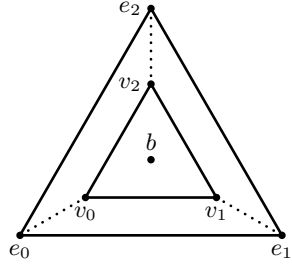
\begin{figure}[h]
    \begin{tikzpicture}[
      x=1cm,y=1cm,scale=2,>=stealth,
      md/.style={line width=1pt,line cap=butt},
      mdover/.style={line width=2pt,color=white,double=black,
        double distance=1pt,line cap=butt},
      thin/.style={md,line width=.5pt},
      thinover/.style={mdover,line width=1.5pt,double distance=.5pt},
      md,
    ]
    \def\c#1{coordinate(#1) node[text=blue]{\tiny#1}} % for debug
    \def\c#1{coordinate(#1)} % for production
    \footnotesize
    \fill \c{b} circle(.7pt) node[above]{$b$};
    \foreach \i/\l in {0/below,1/below,2/left} {
      \fill (-150+\i*120:1) \c{e\i} circle(.7pt) node[\l]{$e_\i$}; }
    \foreach \i/\l in {0/below,1/below,2/left} {
      \fill (barycentric cs:e\i=1,b=1) \c{v\i} circle(.7pt) node[\l]{$v_\i$}; }
    \def\tet#1{ \draw (#10)--(#11)--(#12)--cycle; }
    \tet{e} \tet{v}
    \foreach \i in {0,1,2} { \draw[dash pattern=on 0.1pt off 3pt,line cap=round] (v\i)--(e\i); }
    \end{tikzpicture}
    \caption{A top-dimensional subdivision of~$\Delta^p$ for $p=2$.}
    \label{figure:top-dim-subdivision}
  \end{figure}

  Choose a subdivision $X'$ of the simplicial complex $X$ such that $f\colon X' \to P'$ is simplicial, by applying Lemma~\ref{lemma:simplicial-cellular-subdivision}.
  By Theorem~\ref{theorem:simplicial-cellular-transversality-submanifold}, $\bigcup_\sigma f^{-1}(\sigma_0)$ is a regular neighborhood of the submanifold $\bigcup_\sigma Y_\sigma = \bigcup_\sigma f^{-1}(\hat\sigma)$ in~$X$.
  So the triangulation $X'$ of $X$ induces a triangulation of the exterior $E = X\sm \bigcup_\sigma f^{-1}(\inte \sigma_0)$ of the submanifold $\bigcup_\sigma Y_\sigma$ in~$X$.

  Define a simplicial map $r_0\colon \Delta^p \sm \inte[v_0,\ldots,v_p] \to \partial\Delta^p$ by
  $v_j\mapsto e_j$ and $e_j \mapsto e_j$.
  It is a retraction onto $\partial\Delta^p$, called a ``pseudo-radial'' projection.
  Apply $r_0$ to $\sigma\sm \inte\sigma_0$ for each $p$-simplex $\sigma$ of $P$ to define a simplicial-cellular retraction $r\colon P\sm \bigcup_\sigma \inte \sigma_0 \to P^{(p-1)}$, where $P^{(p-1)}$ is the $(p-1)$-skeleton of~$P$.

  Since $f$ sends $E$ to $P\sm \bigcup_\sigma \inte \sigma_0$, the composition $g = r\circ (f|_E) \colon E\to P^{(p-1)}$ is defined.
  For each $(p-1)$-simplex $\tau$ of $P$, let $Z_\tau = g^{-1}(\hat\tau) \subset E$.
  By Theorem~\ref{theorem:simplicial-cellular-transversality-submanifold}~\ref{item:transversality-submanifold}, $Z_\tau$ is a properly embedded $(n-p+1)$-submanifold.
  Note that
  \[
    \partial Z_\tau = Z_\tau \cap \partial E = f^{-1}r^{-1}(\hat\tau) \cap \biggl(\bigcup_\sigma f^{-1}(\partial\sigma_0)\biggr) = \bigcup_\sigma f^{-1}(r^{-1}(\hat\tau) \cap \partial \sigma_0).
  \]
  Let $d_i\sigma_0 = \phi_\sigma[v_0,\ldots,\hat v_i,\ldots,v_p] \subset \partial\sigma_0$ be the $i$th facet of~$\sigma_0$.
  Then $r^{-1}(\hat\tau) \cap \partial\sigma_0$ is the set of barycenters $\widehat{d_i\sigma_0}$ of facets $d_i\sigma_0$ that $r$ sends to~$\hat\tau$.
  It follows that $\partial Z_\tau$ is the disjoint union of submanifolds of the form $f^{-1}(\widehat{d_i\sigma_0})$.
  Each $f^{-1}(\widehat{d_i\sigma_0})$ is a parallel copy of $Y_\sigma=f^{-1}(\hat\sigma)$ taken along the framing of $Y_\sigma$, by Theorem~\ref{theorem:simplicial-cellular-transversality-submanifold}~\ref{item:simplicial-cellular-transversality}.
  Since $\sigma_0$ has $p+1$ facets $d_i\sigma_0$, $i=0,\ldots,p$, $\partial Z_\tau$ contains at most $p+1$ parallel copies of~$Y_\sigma$.
  Thus $\partial Z_\tau = \bigcup_\sigma k_{\tau\sigma} Y_\sigma$ with $0\le k_{\tau\sigma} \le p+1$.
  (Here $k_{\tau\sigma}$ may be greater than one since multiple facets of $\sigma$ may be identified with~$\tau$ in the simplicial-cell complex~$P$.)

  Suppose that $X$ is oriented.
  As aforementioned, $Y_\sigma$ is oriented so that the product orientation on the  regular neighborhood $Y_\sigma\times\sigma_0 \subset X$ agrees with the orientation of~$X$.
  Also, by Theorem~\ref{theorem:simplicial-cellular-transversality-submanifold}~\ref{item:simplicial-cellular-transversality}, $Z_\tau$ has a regular neighborhood $Z_\tau\times\tau_0$ in $E$ where $\tau_0$ is a $(p-1)$-simplex neighborhood of $\hat\tau$ in $\inte\tau$.
  Orient $Z_\tau$ so that the product orientation on $Z_\tau\times\tau_0$ agrees with the orientation of~$X$, and orient $d_i\sigma_0 \subset \partial\sigma_0$ by the boundary orientation on~$\partial\sigma_0$.
  Note that when $r$ sends $d_i\sigma_0$ to $\tau$, $d_i\sigma = \pm \tau$ in the chain complex $C_*(P)$, and the sign is $+$ if and only if $r\colon d_i\sigma_0 \to \tau$ is orientation preserving. 
  Thus the boundary component $f^{-1}(\widehat{d_i\sigma_0})$ of $Z_\tau$ is a copy of $Y_\sigma$ (resp.\ $-Y_\sigma$) as an oriented manifold if and only if $d_i\sigma = \tau$ (resp.\ $d_i\sigma = -\tau$) in $C_*(P)$.
  From this, it follows that $\partial Z_\tau = d_{\tau\sigma} Y_\sigma$ algebraically if we write $\partial\sigma = \sum_\tau d_{\tau\sigma}\tau$ in~$C_*(P)$.

  When $X$ is unoriented, the argument in the above paragraph shows a weaker conclusion that $k_{\tau\sigma} \equiv d_{\tau\sigma} \bmod 2$.
\end{proof}

\begin{proof}[Proof of Theorem~\ref{theorem:simplicial-cellular-transversality-cobordism}~\ref{item:cobordism-complexity}]
  Note that the above construction of $P'$ subdivides each $p$-simplex $\sigma$ using a fixed subdivision of~$\Delta^p$.
  Thus the number of simplices in the subdivision of $\sigma$ is a constant $N=N(p)$ which is determined by $p=\dim P \le n$.
  Therefore, if $X$ has $\Delta(X)$ simplices, then the subdivision $X'$ of $X$ used above has at most $C\cdot \Delta(X)$ simplices, where $C=C(n)$ is a constant depending only on~$n$, in light of Lemma~\ref{lemma:simplicial-cellular-subdivision}.
  The exterior $E$ of $\bigcup_\sigma Y_\sigma$ in $X$ is a subcomplex of $X'$, so it has at most $C\cdot \Delta(X)$ simplices too.
  Now, apply Theorem~\ref{theorem:simplicial-cellular-transversality-submanifold}~\ref{item:transversality-linear-complexity} to the map $g\colon E\to P^{(p-1)}$ used above, to obtain a triangulation of the disjoint union $\bigcup_\tau Z_\tau$ with at most $C' \cdot C\cdot \Delta(X)$ simplices, where $C'=C'(n)$ is a constant depending only on~$n$.
\end{proof}

For later use, we record some properties of the cell complex structures we used in the above proofs, as a lemma below.

\begin{lemma}
  \label{lemma:transversality-cell-structures}
  Suppose that $f\colon X\to P$ is as in Theorem~\ref{theorem:simplicial-cellular-transversality-cobordism}, with $\dim X = n$, $\dim P = p$.
  Then the following hold.

  \begin{enumerate}
    \item\label{item:cell-structure-Y_sigma}
    The $(n-p)$-manifold $\bigcup_\sigma Y_\sigma$ has a cell complex structure with at most $C(n) \cdot \Delta(X)$ cells such that each cell has at most $C(n)$ vertices.
    \item\label{item:cell-structure-Z_tau}
    The $(n-p+1)$-manifold $\bigcup_\tau Z_\tau$ has a cell complex structure with at most $C(n)\cdot\Delta(X)$ cells each of which has at most $C(n)$ vertices.
    Whenever $Y_\sigma$ is a boundary component of\/ $Z_\tau$, $Y_\sigma$ is a cell subcomplex.
    \item\label{item:cell-structure-X}
    The manifold $X$ has a cell complex structure with at most $C(n)\cdot\Delta(X)$ cells each of which has at most $C(n)$ vertices.
    The product cell complex structure of the regular neighborhood $\bigcup_\sigma Y_\sigma\times \sigma_0$ of\/ $\bigcup_\sigma Y_\sigma$ is a cell subcomplex of~$X$.
  \end{enumerate}
\end{lemma}

In Lemma~\ref{lemma:transversality-cell-structures}~\ref{item:cell-structure-X}, the product cell complex structure means $\{$cells of $Y_\sigma\}\times\{$simplices of $\sigma_0\}$, viewing $\sigma_0\cong \Delta^p$ as the standard simplicial complex.

Lemma~\ref{lemma:transversality-cell-structures}~\ref{item:cell-structure-Y_sigma} follows from the proof of Theorem~\ref{theorem:simplicial-cellular-transversality-submanifold}.
Lemma~\ref{lemma:transversality-cell-structures}~\ref{item:cell-structure-Z_tau} and~\ref{item:cell-structure-X} follows from the proof of Theorem~\ref{theorem:simplicial-cellular-transversality-cobordism}.
(Note that when the proof of Theorem~\ref{theorem:simplicial-cellular-transversality-cobordism} uses Lemma~\ref{lemma:simplicial-cellular-subdivision} to obtain the simplicial subdivision $X'$ of $X$, it actually gives a cell complex structure of $X$ with the desired properties, from which $X'$ is obtained by the pulling triangulation.)

\section{Baumslag-Dyer-Heller acyclic groups and quantitative chain homotopy}
\label{section:BDH-preliminaries}

In~\cite{Baumslag-Dyer-Heller:1980-1}, Baumslag, Dyer and Heller introduced a functor $\cA\colon \{\mathrm{groups}\} \to \{\mathrm{groups}\}$ which assigns to each group $G$ an acyclic group $\cA(G)$ endowed with a functorial injection $G\hookrightarrow \cA(G)$.
In this section, we discuss some preliminary facts on the Baumslag-Dyer-Heller acyclic groups.

\subsubsection*{A direct limit expression of $\cA(G)$}

The acyclic group $\cA(G)$ is given as a direct limit $\cA(G) = \varinjlim_{k} \bA^k(G)$.
Here, $\bA\colon \{\mathrm{groups}\} \to \{\mathrm{groups}\}$ is a functor defined by a ``mitosis'' construction described in~\cite[Section~5]{Baumslag-Dyer-Heller:1980-1} (see also~\cite[Definition~5.1]{Cha:2014-1}).
For each $G$, $\bA(G)$ is equipped with a functorial injection $G\hookrightarrow \bA(G)$.
The groups $\bA^k(G)=\bA(\bA^{k-1}(G))$ are obtained by iterated application, and
the induced injections $\bA^{k-1}(G) \hookrightarrow \bA(\bA^{k-1}(G)) = \bA^k(G)$
form a direct system
\[
  G \hookrightarrow \bA(G) \hookrightarrow \bA^2(G) \hookrightarrow \cdots.
\]
The injection $G\hookrightarrow \cA(G) = \varinjlim_{k} \bA^k(G)$ is the limit of the injections $G\hookrightarrow \bA^k(G)$.

\subsubsection*{Quantitative understanding of acyclicity}

Baumslag, Dyer and Heller proved that the injection $G\hookrightarrow \bA^k(G)$ induces zero homomorphisms $H_i(G;\bF) \to H_i(\bA^k(G);\bF)$ in the range $i\le n$ when $\bF$ is a field.
The acyclicity of $\cA(G)$ is a consequence of this.
In~\cite{Cha:2014-1}, a quantitative chain-level understanding of this homological behavior is given.
See also~\cite{Lim:2022-1} for a subsequent study.
We first recall notations and terms that are needed to state the result.

For a discrete group $G$, denote by $BG$ the simplicial classifying space, which is defined by the bar construction (or the nerve construction) for~$G$.
Denote by $\Z BG_*$ the Moore chain complex of the simplicial set~$BG$.
For each $k$, $\Z BG_k$ is the free abelian group generated by $k$-simplices of~$BG$, including degenerated ones.
The geometric realization of $BG$ has a natural CW-complex structure, whose cells are in 1-1 correspondence with non-degenerated simplices of~$BG$.
As an abuse of notation, we denote the geometric realization by the same symbol~$BG$.
Let $C_*(BG)$ be the cellular chain complex of~$BG$.
Readers not familiar with those in this paragraph may refer to a quick review in~\cite[Appendix~A]{Cha:2014-1}.

The chain complexes $\Z BG$ and $C_*(BG)$ are based, i.e.,\ equipped with the basis consisting of simplices and cells.
In a based chain complex, define the \emph{diameter} of a chain $u$ by $d(u) = \sum_\sigma |a_\sigma|$ if $u$ is written as a linear combination $u=\sum_\sigma a_\sigma \sigma$ of basis elements~$\sigma$ ($a_\sigma\in \Z$).

In this paper, we often use chain maps $\phi\colon C_* \to D_*$ and chain homotopies $P\colon C_*\to D_{*+1}$ defined in a given range of dimensions $*\le n$ only.
We will call them \emph{partial chain maps} and \emph{partial chain homotopies}.
When $C_*$ and $D_*$ are based chain complexes, define the \emph{norm of $\phi$ in dimension${}\le n$} to be $\|\phi\| = \sup_\sigma d(\phi(\sigma))$ where $\sigma$ varies over basis elements in~$C_*$ for $*\le n$.
Similarly, define $\|P\| = \sup_\sigma d(P(\sigma))$.

Now, we can state the aforementioned quantitative result in~\cite{Cha:2014-1}.

\begin{theorem}[{\cite[Theorem~5.2]{Cha:2014-1}}]
  \label{theorem:quantitative-chain-homotopy-moore}
  For each $n>0$, there is a partial chain null-homotopy $\Phi \colon \Z BG_* \to \Z B\bA^n(G)_{*+1}$ defined in dimension${}\le n=\dim M$ for the inclusion-induced chain map $i\colon \Z BG_* \hookrightarrow \Z B\bA^n(G)_*$, i.e., $\Phi \partial + \partial \Phi = i$ on $\Z BG_*$ for $* \le n$.
  In addition, the norm $\|\Phi\|$ in dimension${}\le n$ satisfies $\|\Phi\| \le C(n)$ for some constant $C(n)$ depending only on $n$, independent of~$G$.
\end{theorem}

In this paper, we need the following cellular version for $C_*(BG)$, which we can obtain as a consequence of Theorem~\ref{theorem:quantitative-chain-homotopy-moore}.

\begin{corollary}
  \label{corollary:quantitative-chain-homotopy-cellular}
  For each $n>0$, there is a partial chain null-homotopy $P\colon C_*(BG) \to C_{*+1}(B\bA^n(G))$ in dimension${}\le n$ satisfying the conclusion of Theorem~\ref{theorem:quantitative-chain-homotopy-moore}:
  $P\partial + \partial P = j$ where $j\colon C_*(BG) \hookrightarrow C_*(B\bA^n(G))$ is inclusion-induced, and $\|P\| \le C(n)$ for some constant~$C(n)$ depending only on $n$, independent of~$G$.
\end{corollary}

\begin{proof}
  It is known that $C_*(BG)$ is isomorphic to a quotient $\Z BG_* / D_*$ of the Moore chain complex, where $D_*$ is the subcomplex of $\Z BG_*$  generated by the degenerated simplices.
  Moreover, the quotient map $\pi = \pi_G\colon \Z BG_* \to \Z BG_* / D_* = C_*(BG)$ is a chain homotopy equivalence, which has a chain homotopy inverse $g = g_G\colon C_*(BG) \to \Z BG_*$ such that $\pi_G \circ g_G = \id_{C_*(BG)}$~\cite[Normalization Theorem~VIII.6.1]{MacLane:1995-1}.

  Define a chain homotopy $P\colon C_*(BG) \to C_{*+1}(B\bA^n G)$ by $P=\pi_{\bA^n(G)} \circ \Phi \circ g_G$.
  We have $P\partial + \partial P = \pi_{\bA^n(G)} \circ i \circ g_G$.
  Note that $\pi_{\bA^n(G)} \circ i = j \circ \pi_G$, i.e.,\ the following diagram commutes:
  \[
    \begin{tikzcd}[sep=large]
      \Z BG_* \ar[r,hook,"i"] \ar[d,"\pi_G"'] & \Z B\bA^n(G)_* \ar[d,"\pi_{\bA^n(G)}"]
      \\
      C_*(BG) \ar[r,hook,"j"'] & C_*(B\bA^n(G))
    \end{tikzcd}
  \]
  From this, it follows that $P\partial + \partial P = j \circ \pi_G \circ g_G = j$, i.e.,\ $P$ is a partial chain null-homotopy for~$j$ in dimension${}\le n$.

  Inspecting the construction of the chain homotopy inverse $g$ described in~\cite[Proof of Theorem~VIII.6.1]{MacLane:1995-1}, it is verified that the norm of $g$ in dimension${}\le n$ is bounded by $(2n+3)^{n+1}$.
  (Using notations used in ~\cite[p.~236--237]{MacLane:1995-1}, $g\colon C_*(BG) \to \Z BG_*$ is induced by a chain map $h\colon \Z BG_* \to \Z BG_*$, where $h=h_0 h_1 \cdots h_{n+1}$ in dimension${}\le n$, $h_k = 1 - t_k\partial - \partial t_k$, $t_k$ is a chain homotopy with $\|t_k\| = 1$;
  since $\|\partial\| \le n+1$, the assertion $\|g\| \le (2n+3)^{n+1}$ follows.)

  Since $\|\pi\| = 1$ for the projection~$\pi$, it follows that
  \[
    \|P\| \le \|\pi\| \cdot \|\Phi\| \cdot \|g\| \le (2n+3)^{n+1} \cdot \|\Phi\|
  \]
  in dimension${}\le n$.
\end{proof}

\section{Quantitative PL bordism}
\label{section:quantitative-pl-bordism}

In Section~\ref{subsection:proof-main-quantitative-bordism}, we give a proof of Theorem~\ref{theorem:main-quantitative-bordism}\@.
A key step of the proof is Proposition~\ref{proposition:bordism-toward-lower-skeleton}, which constructs a quantitative bordism from a chain homotopy we described in Corollary~\ref{corollary:quantitative-chain-homotopy-cellular}.
We prove Proposition~\ref{proposition:bordism-toward-lower-skeleton} in Section~\ref{subsection:bordism-from-chain-null-homotopy}.

In this section, we work in the category of PL manifolds.
A triangulation of a manifold is always a combinatorial triangulation.

\subsection{Proof of Theorem~\ref{theorem:main-quantitative-bordism}}
\label{subsection:proof-main-quantitative-bordism}

For brevity, we will write ``$N\le C(n)$'' or ``$N\le C(n)\cdot \Delta(M)$'' when the inequality holds for some constant $C(n)$ which depends only on~$n$.
We will use the same symbol $C(n)$ for different constants appearing on different occasions.

The goal of this section is to prove the following:

\begin{theorem-named}[Theorem~\ref{theorem:main-quantitative-bordism}]
  Let $M$ be a triangulated closed $n$-manifold over the acyclic group $\Gamma = \cA(G)$ where $G$ is an arbitrary group.
  Then there exists a triangulated bordism $W$ over $\Gamma$ from $M$ to a trivial end such that $\Delta(W) \leq C(n) \cdot \Delta(M)$.
  In addition, $W$ is an oriented bordism if $M$ is oriented.
\end{theorem-named}

\begin{proof}
  First, we may assume that the given $M\to B\cA(G)$ is the composition of a map $\phi \colon M\to BG$ and the functorial inclusion $BG\hookrightarrow B\cA(G)$.
  To verify this, recall from Section~\ref{section:BDH-preliminaries} that $\cA(G)$ is the direct limit of the sequence $G\hookrightarrow \bA(G) \hookrightarrow \bA^2(G) \hookrightarrow \cdots$.
  Since $\pi_1(M)$ is finitely generated, $\pi_1(M) \to \cA(G)$ factors through $\bA^k(G)$ for some $k$, and thus we may assume that $M\to B\cA(G)$ factors through $B\bA^k(G)$.
  Since our goal is to construct a bordism over $\cA(G)$, we can replace $G$ with $\bA^k(G)$, to assume that $M\to B\cA(G)$ factors through a map $\phi\colon M\to BG$.
  This shows the claim.

  In addition, once we have $\phi\colon M\to BG$, we may assume that $\phi$ is simplicial-cellular, by applying the following result in~\cite{Cha:2014-1} to~$\phi$.

  \begin{theorem}[{\cite[A special case of Theorem~3.7]{Cha:2014-1}}]
    \label{theorem:simplicial-cellular-approximation}
    Suppose that $\phi\colon K\to BG$ is a map of a simplicial complex~$K$.
    Then $\phi$ is homotopic to a simplicial-cellular map.
  \end{theorem}

  Note that $\phi$ maps $M$ into the $n$-skeleton $BG^{(n)}$ since $\dim M=n$ and $\phi$ is simplicial-cellular.
  Let $M_n=M$ and $\phi_n = \phi\colon M_n \to BG^{(n)}$.

  Inductively, suppose that we have obtained a closed $n$-manifold $M_p$ equipped with a simplicial-cellular map $\phi_p \colon M_p \to BG^{(p)}$, $0 < p \le n$.
  Let $i\colon C_*(BG) \to C_*(B\bA^n(G))$ be the inclusion-induced chain map.
  By Corollary~\ref{corollary:quantitative-chain-homotopy-cellular}, there is a partial chain null-homotopy $P\colon C_*(BG) \to C_{*+1}(B\bA^n(G))$ for $i$, defined in dimension${}\le n=\dim M$, which satisfies $\|P\| \le C(n)$.
  This enables us to apply Proposition~\ref{proposition:bordism-toward-lower-skeleton} stated below to
  $(M,L,K,f) = (M_p,BG^{(p)},B\bA^n(G),\phi_p)$:

  \begin{proposition}
    \label{proposition:bordism-toward-lower-skeleton}
    Let $f\colon M\to L$ be a simplicial-cellular map of a closed triangulated $n$-manifold $M$ into a simplicial-cell complex $L$ with $\dim L=p$, which is a subcomplex of another simplicial-cell complex~$K$.
    Suppose that there is a partial chain null-homotopy $P \colon C_*(L) \to C_{*+1}(K)$, defined in dimension $p$ and $p-1$, for the chain map $i \colon C_*(L)\to C_*(K)$ induced by the inclusion $L\hookrightarrow K$, i.e.,\ $P\partial + \partial P = i$ on $C_p(L)$ and $C_{p-1}(L)$.
    Suppose $\|P\| < \infty$.
    Then there exist a closed triangulated $n$-manifold $N$ over the $(p-1)$-skeleton $K^{(p-1)}$ and a triangulated bordism $W$ over~$K^{(p+1)}$ from $M$ to~$N$.
    \[
      \begin{tikzcd}[column sep={15mm,between origins},row sep=large]
        N \ar[r,hook] \ar[d] & W \ar[rrrd] & M \ar[l,hook'] \ar[d,crossing over]
        \\
        K^{(p-1)} \ar[rr,hook] & & K^{(p)} \ar[rr,hook] & & K^{(p+1)}
      \end{tikzcd}
    \]
    In addition, $W$ and $N$ have triangulations with linearly bounded complexity:
    $\Delta(W) \le C\cdot \Delta(M)$,
    and consequently $\Delta(N) < C\cdot \Delta(M)$, where $C = C(n, \|P\|)$ is a constant depending only on $n=\dim M$ and the norm~$\|P\|$.
  \end{proposition}

  We postpone the proof of Proposition~\ref{proposition:bordism-toward-lower-skeleton} and continue the proof of Theorem~\ref{theorem:main-quantitative-bordism}\@.
  In our case, from Proposition~\ref{proposition:bordism-toward-lower-skeleton}, it follows that there is a closed triangulated $n$-manifold $M_{p-1}$ equipped with a map $\phi_{p-1} \colon M_{p-1} \to B\bA^n(G)^{(p-1)}$ and a triangulated bordism $W_p$ over $B\cA(G)$ from $M_p$ to $M_{p-1}$ such that $\Delta(W_p) \le C(n)\cdot \Delta(M_p)$.
  We may assume that $\phi_{p-1}$ is simplicial-cellular by Theorem~\ref{theorem:simplicial-cellular-approximation}.
  By replacing $G$ with $\bA^n(G)$, we may assume that $M_{p-1}$ is over~$BG^{(p-1)}$ via $\phi_p$ so that we can continue the induction.

  By the above inductive construction for $p=n$, $n-1$,~\ldots,~$1$, now we have the manifolds $M=M_n,M_{n-1},\ldots,M_0$ and bordisms $W_n,\ldots,W_1$ between them.
  Concatenate the bordisms $W_p$ to obtain
  \[
    W=W_{n}\cupover{M_{n-1}} W_{n-1}\cupover{M_{n-2}} \cdots \cupover{M_1} W_0
  \]
  and let $N=M_0$.
  The bordism $W$ is over $\cA(G)$, from $M_n=M$ to~$N=M_0$.
  Since $N$ is equipped with $\phi_0\colon N=M_0 \to BG^{(0)}=\{*\}$, $N$ is over $\cA(G)$ via a constant map.
  Also, we have $\Delta(W) \le \Delta(W_n)+\cdots+\Delta(W_1) \le C(n)\cdot \Delta(M)$.
  This completes the proof of Theorem~\ref{theorem:main-quantitative-bordism}, modulo the proof of Proposition~\ref{proposition:bordism-toward-lower-skeleton} which we give below.
\end{proof}

\subsection{Construction of bordism from a chain null-homotopy}
\label{subsection:bordism-from-chain-null-homotopy}

\begin{proof}[Proof of Proposition~\ref{proposition:bordism-toward-lower-skeleton}]
  Suppose that $f\colon M\to L$ is simplicial-cellular, where $M$ is a closed triangulated $n$-manifold and $L$ is a simplicial-cell complex.
  In addition, suppose that $L$ is a subcomplex of a simplicial-cell complex $K$ and $P \colon C_*(L) \to C_{*+1}(K)$ is a partial chain null-homotopy defined in dimension $p$ and $p-1$ for the inclusion-induced $i \colon C_*(L)\to C_*(K)$ such that $\|P\| < \infty$, where $p = \dim L$.
  Our goal is to construct a bordism $W$ over $K^{(p+1)}$, from $M$ to another $n$-manifold $N$ which is over $K^{(p-1)}$, such that $\Delta(W) \le C(n,\|P\|)\cdot \Delta(M)$.
  We will focus on the oriented case, and discuss modifications for the unoriented case at the end of the proof.

  If $n<p$, then the conclusion is trivial since the given $f$ sends $M$ to~$L^{(p-1)}\subset K^{(p-1)}$.
  So, suppose $n\ge p$.

  We will construct a desired bordism $W$ from several simpler bordisms.
  First, we will use the algebraic data from the chain null-homotopy~$P$ to construct a bordism over $K^{(p)}$, which we denote by $V_\sigma$ below.
  For a $p$-simplex $\sigma$ of~$L$ and a $(p-1)$-simplex $\tau$ of~$L$, write $\partial\sigma$, $P\sigma$ and $P\tau$ as linear combinations of simplices
  \[
    \partial\sigma = \sum_\tau d_{\tau\sigma} \cdot \tau, \quad P\sigma = \sum_\mu r_{\mu\sigma} \cdot \mu, \quad P\tau = \sum_\eta s_{\eta\tau}  \cdot\eta
  \]
  where $\tau$, $\mu$ and $\eta$ in the sums vary over $(p-1)$, $(p+1)$ and $p$-simplices of $L$, $K$ and $K$, respectively, and $d_{\tau\sigma}, r_{\mu\sigma}, s_{\eta\tau}\in \Z$.
  Evaluate $P\partial + \partial P - i = 0$ on~$\sigma$, to obtain an equation
  \begin{equation}
    \label{equation:chain-homotopy-equation}
    \Biggl( \sum_{\eta,\tau} s_{\eta\tau} d_{\tau\sigma} \cdot \eta + \sum_\mu r_{\mu\sigma} \cdot \partial\mu \Biggr) - \sigma = 0
  \end{equation}
  which holds in~$C_p(K)$, viewing $\sigma$ as a simplex in~$K$ via $L\subset K$.

  For each~$\eta$, let $\Delta_\eta^p$ be an associated copy of the standard $p$-simplex.
  Similarly, for each~$\mu$, let $\Delta_\mu^{p+1}$ be a copy of the standard $(p+1)$-simplex.
  Let
  \begin{equation}
    \label{equation:partial-V_sigma}
    \partial_- V_\sigma = \Delta^p_\sigma,
    \quad
    \partial_+ V_\sigma = \Biggl( \bigcup_{\eta,\tau} s_{\eta\tau} d_{\tau\sigma} \cdot \Delta_\eta^p \Biggr) \cup \Biggl( \bigcup_\mu r_{\mu\sigma} \cdot \partial\Delta_\mu^{p+1} \Biggr)
  \end{equation}
  and consider the complex $\partial_+ V_\sigma \cup -\partial_- V_\sigma$, which is an oriented $p$-manifold with boundary.
  Here each simplex is oriented by its sign in the expression~\eqref{equation:partial-V_sigma}.
  Viewing the left-hand side of~\eqref{equation:chain-homotopy-equation} as a formal sum of $(\pm1)$-signed $p$-simplices of $K$, the $p$-simplices of $\partial_+ V_\sigma \cup -\partial_- V_\sigma$ are in 1-1 correspondence with the terms in~\eqref{equation:chain-homotopy-equation}.
  Take the product cobordism $(\partial_+ V_\sigma \cup -\partial_- V_\sigma)\times I$, and for each canceling pair in \eqref{equation:chain-homotopy-equation}, attach a 1-handle $\Delta^p \times I$ that realizes the algebraic cancellation, i.e.,\ identify $\Delta^p \times 0$ and $\Delta^p \times 1$ with the two paired $p$-simplices in $(\partial_+ V_\sigma \cup -\partial_- V_\sigma) \times 1 \subset (\partial_+ V_\sigma \cup -\partial_- V_\sigma) \times I$ which are canceled in~\eqref{equation:chain-homotopy-equation}.
  Let $V_\sigma$ be the resulting $(p+1)$-manifold.
  Note that $\partial V_\sigma$ is of the form $\partial V_\sigma = \partial_+V_\sigma \cup \partial_0 V_\sigma \cup -\partial_-V_\sigma$, where the belt sphere $\partial\Delta^p \times I$ of each 1-handle is contained in~$\partial_0 V_\sigma$.

  The characteristic maps of $\eta$, $\mu$ and $\sigma$ induce a map $\partial_+ V_\sigma \cup - \partial_- V_\sigma \to K^{(p)}$.
  It extends to the 1-handles, and thus $V_\sigma$ is over~$K^{(p)}$.
  In addition, $\partial_0 V_\sigma$ is over $K^{(p-1)}$ since the boundary of each $p$-simplex of $\partial_+ V_\sigma \cup -\partial_- V_\sigma$ is sent to~$K^{(p-1)}$.

  By Theorem~\ref{theorem:simplicial-cellular-transversality-submanifold}, the inverse image $Y_\sigma = f^{-1}(\hat\sigma)$ is a closed $(n-p)$-submanifold in~$M$.
  Let $V = (-1)^{n-p} \bigcup_\sigma Y_\sigma \times V_\sigma$ and $\partial_\bullet V = (-1)^{n-p} \bigcup_\sigma Y_\sigma\times \partial_\bullet V_\sigma$ for $\bullet=+,-,0$.
  Using $\partial(Y_\sigma \times X) = (-1)^{n-p} Y_\sigma \times \partial X$, we have $\partial V = \partial_+V \cup \partial_0V \cup -\partial_-V$.
  Also, $V$ is over $K^{(p)}$ via $Y_\sigma\times V_\sigma \to V_\sigma \to K^{(p)}$.
  Since $\partial_0 V_\sigma$ is over $K^{(p-1)}$, $\partial_0 V$ is over~$K^{(p-1)}$.

  To construct the next building block, recall from Theorem~\ref{theorem:simplicial-cellular-transversality-cobordism} that for each $\tau$, there is an associated $(n-p+1)$-submanifold $Z_\tau$ in~$M$.
  We slightly modify $Z_\tau$, as detailed below.

  Since $X$ is oriented, the oriented manifold $Z_\tau$ satisfies $\partial Z_\tau = \bigcup_\sigma d_{\tau\sigma} \cdot Y_\sigma$ algebraically, by Theorem~\ref{theorem:simplicial-cellular-transversality-cobordism}~\ref{item:inverse-image-oriented-cobordism}.
  For each algebraically canceling pair of boundary components of $Z_\tau$ which are copies of $Y_\sigma$, attach a copy of $Y_\sigma\times I$ to eliminate the pair.
  Abusing the notation, denote the resulting manifold by~$Z_\tau$.
  Now $\partial Z_\tau$ is (geometrically) equal to $\bigcup_\sigma d_{\tau\sigma} \cdot Y_\sigma$ as an oriented manifold.
  For later use, recall that the original $Z_\tau$ is equipped with a cell complex structure with at most $C(n)\cdot\Delta(X)$ cells each of which has at most $C(n)$ vertices, by Lemma~\ref{lemma:transversality-cell-structures}~\ref{item:cell-structure-Z_tau}.
  Our modified $Z_\tau$ has the same property, since the total number of cells in the additional copies of $Y_\sigma\times I$ is at most the number of cells of $\partial Z_\tau \times I$, which is at most 3 times the number of cells of~$Z_\tau$.
  
  Let
  \begin{equation}
    \label{equation:definition-of-U}
    U = \Biggl( \bigcup_{\eta,\tau} s_{\eta\tau} \cdot Z_\tau \times \Delta_\eta^p \Biggr)
    \cup \Biggl( \bigcup_{\mu,\sigma} (-1)^{n-p} r_{\mu\sigma} \cdot Y_\sigma \times \Delta_\mu^{p+1} \Biggr).
  \end{equation}
  Then we have $\partial U = \partial_+U \cup -\partial_-U$ where
  \[
    \begin{aligned}
      \partial_-U &= \Biggl( \bigcup_{\eta,\tau,\sigma} s_{\eta\tau} d_{\tau\sigma}\cdot Y_\sigma \times \Delta_\eta^p \Biggr)
      \cup \Biggl( \bigcup_{\mu,\sigma} r_{\mu\sigma} \cdot Y_\sigma \times \partial\Delta_\mu^{p+1} \Biggr),
      \\
      \partial_+U &= (-1)^{n-p+1} \Biggl( \bigcup_{\eta,\tau} s_{\mu\tau} \cdot Z_\tau \times \partial\Delta_\eta^p \Biggr).
    \end{aligned}
  \]
  The $(p+1)$-manifold $U$ is over $K^{(p+1)}$ via the maps $Z_\tau \times \Delta_\eta^p \to \Delta_\eta^p \to K^{(p)}$ and $Y_\sigma \times \Delta_\mu^{p+1} \to \Delta_\mu^{p+1} \to K^{(p+1)}$.
  Note that $\partial_+U$ is over $K^{(p-1)}$ since so is~$\partial\Delta_\eta^p$.
  Also, $\partial_-U = \partial_+V$.

  Concatenate $V$ and $U$ along $\partial_+V = \partial_-U$, to obtain $Z=V \cup U$, which is over $K^{(p+1)}$.
  We have $\partial Z = \partial_+Z \cup -\partial_- Z$ where $\partial_-Z = \partial_-V = \bigcup_\sigma Y_\sigma \times \Delta_\sigma^p$ and $\partial_+Z = \partial_0 V \cup \partial_+ U$.
  See the schematic diagram in Figure~\ref{figure:bordism-Z}.
  Note that $\partial_+Z$ is over~$K^{(p-1)}$.

  \begin{figure}[h]
    \begin{tikzpicture}[
      x=12mm,y=13mm,scale=1,>=stealth,
      md/.style={line width=1pt,line cap=butt},
      mdover/.style={line width=2pt,color=white,double=black,
        double distance=1pt,line cap=butt},
      thin/.style={md,line width=.5pt},
      thinover/.style={mdover,line width=1.5pt,double distance=.5pt},
      md,
    ]
      \def\c#1{coordinate(#1) node[text=blue]{\tiny#1}} % for debug
      \def\c#1{coordinate(#1)} % for production
      \draw (0,0) -- ++(2,0) \c{A} arc(90:-90:.5 and 1)
      -- ++(-2,0) arc(-90:270:.5 and 1);
      \draw[thin] (A) arc(90:270:.5 and 1);
      \draw (A) arc(90:-90:2 and 1);
      \draw (1,-1) node{$V$} ++(2.1,0) node{$U$};
      \draw (0,-1) node{\small $\partial_- V$};
      \draw (1,-2) node[anchor=north]{\small $\partial_0 V$};
      \draw (2,-1) node{\small
        $\begin{array}{l} \partial_+ V \\ = \partial_- U \end{array}$};
      \draw (4,-1) node[anchor=west]{\small $\partial_+ U$};
    \end{tikzpicture}
    \caption{The bordism $Z=V \cupover{\partial_+V = \partial_-U} U$}
    \label{figure:bordism-Z}
  \end{figure}
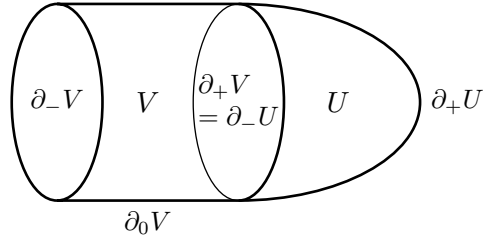

  Now, attach $Z$ to $M\times I$, identifying $\partial_-Z = \bigcup_\sigma Y_\sigma \times \Delta_\sigma^p$ with the canonical regular neighborhood $\bigcup_\sigma Y_\sigma \times \sigma_0$ of $\bigcup_\sigma Y_\sigma$ in $M=M\times 1$, to obtain a cobordism $W=(M\times I) \cup Z$ from $M=M\times 0$ to $N=E \cup_\partial \partial_+Z$ where $E=M \sm \bigcup_\sigma \inte(Y_\sigma \times \sigma_0)$ is the exterior of $\bigcup_\sigma Y_\sigma$ in~$M$.

  To make $W = (M\times I) \cup Z$ over $K^{(p+1)}$, we will use the given $f\colon M\to L\subset K^{(p)}$ and the above $Z\to K^{(p+1)}$.
  The restrictions of $f$ on $\bigcup_\sigma Y_\sigma \times \sigma_0$ and $Z\to K^{(p+1)}$ on $\partial_-Z = \bigcup_\sigma Y_\sigma \times \Delta_\sigma^p$ do not agree since $\sigma_0\subsetneq\sigma$, but one can deform $f$ slightly by expanding $\sigma_0$ to~$\sigma$.
  More precisely, the identity map of $L$ is homotopic rel $L^{(p-1)}$ to a map which sends $\inte\sigma_0$ onto $\inte\sigma$ for each~$\sigma$.
  Compose the homotopy $L\times I \to L$ with $\phi\times\id \colon M\times I \to L\times I$ to make $M\times I$ over~$L$.
  Then, $M\times I \to L \subset K^{(p)}$ agrees with $Z\to K^{(p+1)}$ on $\bigcup_\sigma Y_\sigma \times \sigma_0 = \partial_-Z$, so that they induce $W \to K^{(p+1)}$.
  Since the exterior $E$ and $\partial_+Z$ are over~$K^{(p-1)}$, $N$ is over $K^{(p-1)}$ as desired.

  It remains to estimate the complexity of~$W$.
  To do this, we will first construct cell complex structures of the building blocks $V$, $U$ and $M\times I$, which are consistent on the intersections so that they define a cell complex structure of~$W$.
  Details are as follows.

  By Lemma~\ref{lemma:transversality-cell-structures}~\ref{item:cell-structure-Y_sigma}, $\bigcup_\sigma Y_\sigma$ has a cell complex structure with at most $C(n)\cdot \Delta(M)$ cells, each of which has at most $C(n)$ vertices.
  By Lemma~\ref{lemma:transversality-cell-structures}~\ref{item:cell-structure-X}, $M$ has a cell complex structure with at most $C(n)\cdot \Delta(M)$ cells such that each cell has at most $C(n)$ vertices and the product regular neighborhood $\bigcup_\sigma Y_\sigma \times \sigma_0$ of $\bigcup_\sigma Y_\sigma$ is a cell subcomplex.
  (Here, recall that $\sigma_0 \subset \inte\sigma$ is a $p$-simplex neighborhood of the barycenter~$\hat\sigma$.)

  Note that for each~$\sigma$, since $\sum_\tau |d_{\tau\sigma}| \le p+1 \le n+1$ and $\sum_\mu |r_{\mu\sigma}|$, $\sum_\eta |s_{\eta\tau}| \le \|P\|$, there are at most $(n+1)\cdot \|P\| + \|P\| +1$ $(\pm1)$-signed $p$-simplices in~\eqref{equation:chain-homotopy-equation}.
  Take the product cell complex structures for $(\partial_+ V_\sigma \cup -\partial_- V_\sigma)\times I$ and for the 1-handles, to obtain a cell complex structure of $V_\sigma$ with at most $C(n,\|P\|)$ cells each of which has at most $C(n)$ vertices.
  Combining this with the above property of the cell complex structure of~$Y_\sigma$, it follows that the product cell complex structure on $V = (-1)^{n-p} \bigcup_\sigma Y_\sigma \times V_\sigma$ has at most $C(n,\|P\|) \cdot \Delta(M)$ cells each of which has at most $C(n)$ vertices.

  By Lemma~\ref{lemma:transversality-cell-structures}~\ref{item:cell-structure-Z_tau}, $\bigcup_\tau Z_\tau$ has a cell structure with at most $C(n) \cdot \Delta(M)$ cells each of which has at most $C(n)$ vertices.
  Use the product cell structure to obtain a cell complex structure of $U$ defined in~\eqref{equation:definition-of-U}.

  The above cell complex structures of $V$ and $U$ agree on $\partial_+V = \partial_-U$ since both are obtained by the product construction.
  Thus they define a cell complex structure of $Z=V\cup U$.
  By Lemma~\ref{lemma:transversality-cell-structures}~\ref{item:cell-structure-X},
  the cell subcomplex structure of $\bigcup_\sigma Y_\sigma\times \sigma_0 \subset M = M\times 1 \subset M\times I$ is equal to the product cell complex structure.
  Therefore the cell complex structures of $M\times I$ and $Z$ agree on $\bigcup_\sigma Y_\sigma\times \sigma_0 = \partial_- Z$ so that they define a cell complex structure of $W=(M\times I) \cup Z$.

  Now, fix an order of vertices of the cell complex structure of $W$ and apply the pulling triangulation in Lemma~\ref{lemma:pulling-triangulation} to triangulate~$W$.
  This triangulates the subcomplexes $M\times I$, $V$ and $U$ as well.

  Note that if $X$ and $Y$ are cell complexes with at most $k$ cells and if each cell has at most $r$ vertices, then $X\times Y$ is a cell complex with $k^2$ cells each of which has at most $r^2$ vertices.
  So, by applying the pulling triangulation in Lemma~\ref{lemma:pulling-triangulation}, $X\times Y$ has a triangulation with $\Delta(X\times Y)\le 2^{(r^2)} k^2$.
  From this and from the above upper bounds of the number of cells and vertices, it follows that
  \[
    \Delta(M\times I) \le C(n)\cdot \Delta(M), \quad \Delta(V) \le C(n,\|P\|)\cdot \Delta(M), \quad \Delta(U) \le C(n)\cdot \Delta(M).
  \]
  Therefore we have
  \[
    \Delta(W) \le \Delta(M\times I) + \Delta(V) + \Delta(U) \le C(n,\|P\|) \cdot \Delta(M).
  \]
  
  Finally, if $X$ is unoriented, then $Y_\sigma$ and $Z_\tau$ are unoriented, and we have $\partial Z_\tau = \bigcup_\sigma k_{\tau\sigma} Y_\sigma$ with $k_{\tau\sigma} \equiv d_{\tau\sigma} \bmod 2$ by Theorem~\ref{theorem:simplicial-cellular-transversality-cobordism}~\ref{item:inverse-image-cobordism}.
  Using this instead of the oriented case boundary $\partial Z_\tau = \bigcup_\sigma d_{\tau\sigma} Y_\sigma$, and by ignoring signs, the above argument proves the conclusion for the unoriented case.
\end{proof}

\begin{remark}
  \label{remark:comparison-with-dim-3}
  A simpler argument which proves Theorem~\ref{theorem:main-quantitative-bordism} for the special case of $n=\dim M=3$ appeared earlier in~\cite{Cha:2014-1}.
  Briefly, to construct bordism which corresponds to our $W_p$ for $0<p<\dim 3$, the arguments in~\cite{Cha:2014-1} rely on the quantitative vanishing of $\Omega_q$ for $0<q<3$, more precisely, the fact that a closed $q$-manifold bounds a linear null-cobordism.
  This is used to quantitatively understand the vanishing of the $E^2$ page $E^2_{p,q}=H_p(\Gamma;\Omega_{3-p})$ of the Atiyah-Hirzebruch spectral sequence in~\cite{Cha:2014-1}.
  Obviously, it does not apply to high dimensions since $\Omega_q\ne 0$ in general.
  In this regard, for the general case of arbitrary dimension, it is essential to use the (partial) chain homotopy in dimension $p$ and $p-1$ in Proposition~\ref{proposition:bordism-toward-lower-skeleton} and the geometric arguments in Section~\ref{subsection:bordism-from-chain-null-homotopy} that build on the simplicial transversality in Section~\ref{section:simplicial-cellular-transversality}.
\end{remark}

\section{Quantitative smooth bordism}
\label{section:quantitative-smooth-bordism}

Recall that the complexity $V(M)$ of a smooth manifold $M$ is defined to be the infimum of the volume of a Riemannian structure on $M$ with bounded local geometry, i.e.,\ $($injectivity radius$) \ge 1$ and $|$sectional curvature$| \le 1$.
If $\partial M$ is nonempty, then we require that the Riemannian structure restricts to a product structure on a collar $\partial M \times I$ where $I$ has unit length.
Note that this implies that $V(\partial M) \le V(M)$.

The goal of this section is to prove the following:

\begin{theorem-named}[Theorem~\ref{theorem:main-quantitative-smooth-bordism}]
  Let $M$ be a smooth $n$-manifold over the Baumslag-Dyer-Heller acyclic group $\cA(G)$, with $G$ arbitrary.
  Then there exists a smooth bordism $W$ over $\cA(G)$ from $M$ to a trivial end such that $V(W) \le C(n) \cdot V(M)$, where $C(n)$ is a constant depending only on~$n$.
  In addition, if $M$ is oriented, then $W$ is an oriented bordism.
\end{theorem-named}

\begin{proof}
  There is a (piecewise smooth) triangulation of $M$ such that $\Delta(M) \le C(n)\cdot V(M)$ and each vertex is contained in at most $C(n)$ simplices, by~\cite[Theorem~3]{Boissonnat-Dyer-Ghosh:2018-1}.
  See also~\cite[Theorem~3.2 in Appendix~A]{Chambers-Dotterrer-Manin-Weinberger:2018-1} and~\cite[Theorem~A.1]{Manin-Weinberger:2023-1}.

  By Theorem~\ref{theorem:main-quantitative-bordism}, there is a PL bordism $W$ over $\Gamma$ from $M$ to a trivial end~$N$, with $\Delta(W) \le C(n)\cdot \Delta(M)$.

  We assert that the PL bordism $W$ is smoothable.
  To show this, recall that the proof of Theorem~\ref{theorem:main-quantitative-bordism} constructs $W$ by stacking bordisms obtained by inductive applications of Proposition~\ref{proposition:bordism-toward-lower-skeleton}.
  In each step, Proposition~\ref{proposition:bordism-toward-lower-skeleton} produces a bordism $W_p$ between $M_p$ and $M_{p-1}$, $p=n,n-1,\ldots, 1$, where $M_n$ is the given~$M$ and $M_0$ is the trivial end~$N$.
  The bordism $W_p$ is obtained from the submanifolds $Y_\sigma$ and $Z_\tau$ in $M_p$ given by Theorems~\ref{theorem:simplicial-cellular-transversality-submanifold} and~\ref{theorem:simplicial-cellular-transversality-cobordism}.
  From the proof of Proposition~\ref{proposition:bordism-toward-lower-skeleton}, it follows that $W_p$ is smoothable if $Y_\sigma$ and $Z_\tau$ are smoothable.

  By Theorem~\ref{theorem:simplicial-cellular-transversality-submanifold}~\ref{item:simplicial-cellular-transversality}, $Y_\sigma$ has a product regular neighborhood $Y_\sigma\times \Delta^p$ in~$M_p$, $\Delta^p$ a $p$-simplex.
  By induction, $M_p$ is smoothable, and thus the codimension zero open submanifold $Y_\sigma \times \inte\Delta^p \cong Y_\sigma \times \R^p$ in $M_p$ is smoothable.
  So, by the product structure theorem for smoothings of PL manifolds~\cite[Theorem~7.7]{Hirsch-Mazur:1974-1}, $Y_\sigma$ is smoothable.
  Since $Z_\tau$ has a product regular neighborhood in $M_p$ by Theorem~\ref{theorem:simplicial-cellular-transversality-cobordism}~\ref{item:inverse-image-cobordism}, the same argument shows that $Z_\tau$ is smoothable.
  This proves the assertion.

  We need one more observation:
  in the triangulation of the PL bordism $W$, each vertex is contained in at most $C(n)$ simplices.
  This follows because our triangulation of $M$ has the same property, since each step of the construction of $W$ in the proof of Theorem~\ref{theorem:main-quantitative-bordism} raises the bound at most by a $C(n)$ factor.

  Now, we invoke the following quantitative smoothing result:

  \begin{theorem}[{\cite[Theorem~A.2]{Manin-Weinberger:2023-1}}]
    If $W$ is a triangulated PL $n$-manifold such that each vertex is contained in at most $L$ simplices, then every smoothing of $W$ has a Riemannian structure with bounded local geometry such that each simplex of $W$ has volume at most~$C(n,L)$.
  \end{theorem}

  In our case, it follows that there is a Riemannian structure on $W$ with bounded local geometry such that the volume of each simplex is at most~$C(n)$.
  Thus $W$ has linearly bounded complexity $V(W) \le C(n)\cdot \Delta(W)$.
  In addition, the volume of the boundary component $N$ of $W$ is at most $C(n)\cdot \Delta(W)$ since $\Delta(N) \le \Delta(W)$.
  It follows that the linear bound $V(W) \le C(n)\cdot \Delta(W)$ remains true after attaching to $W$ a collar $N\times I$ endowed with the product Riemannian structure to satisfy the collar condition.
  Since $\Delta(W) \le C(n) \cdot \Delta(M)$ and $\Delta(M)\le C(n)\cdot V(M)$, this completes the proof.
\end{proof}

\section{Linear bounds for Cheeger-Gromov \texorpdfstring{$\rho$}{rho}-invariants}
\label{section:bounds-cheeger-gromov-rho}

In this section, we discuss how one obtains the linear bound for the Cheeger-Gromov $L^2$ $\rho$-invariant $\rho^{(2)}(M,\phi)$ stated in Theorem~\ref{theorem:main-rho-invariant}, from the existence of a linear bordism in Theorem~\ref{theorem:main-quantitative-bordism}\@.
Essentially we follow the arguments given in~\cite[Section~2]{Cha:2014-1}.

\begin{proof}[Proof of Theorem~\ref{theorem:main-rho-invariant}]

  Let $M$ be a closed oriented PL $(4k-1)$-manifold and $\phi\colon \pi_1(M)\to G$ is an arbitrary group homomorphism.
  We may assume that $M$ is connected, since $\rho^{(2)}(-)$ and $\Delta(-)$ are additive under disjoint union.
  By Theorem~\ref{theorem:main-quantitative-bordism}, there is an oriented PL bordism $W$ from $M$ to a trivial end $N$ over $\cA(\pi_1(M))$ with $\Delta(W) \le C(n)\cdot \Delta(M)$, where $N$ is over $\cA(\pi_1(M))$ via a trivial map.
  By the functoriality of $\cA(-)$, we have the following commutative diagram.
  \[
    \begin{tikzcd}[sep=large]
      \pi_1(M) \ar[r,"\phi"] \ar[d, "i_*"] \ar[rd,hook,"i_{\pi_1(M)}"]
      & G \ar[rd,hook,"i_{G}"]
      \\
      \pi_1(W) \ar[r]
      & \cA(\pi_1(M)) \ar[r,"\cA(\phi)"']
      & \cA(G) = \Gamma
    \end{tikzcd}
  \]

  We need the following facts, which are readily obtained from standard properties of the $\rho$-invariants:

  \begin{enumerate}
    \item Since $i_G$ is injective, $\rho^{(2)}(M,\phi) = \rho^{(2)}(M,i_G\circ\phi)$ by $L^2$-induction~\cite[eq.~(2.3)]{Cheeger-Gromov:1985-1}, \cite[p.~253]{Lueck:2002-1}, \cite[Proposition~5.13]{Cochran-Orr-Teichner:1999-1}.
    \item Since $N$ is over $\Gamma$ via a trivial map and $W$ is a bordism between $M$ and $N$ over $\Gamma$, from the $L^2$-signature defect definition of $\rho^{(2)}(M)$ it follows that
    \[
      \rho^{(2)}(M,i_G\circ\phi) = \sign^{(2)}_\Gamma(W) - \sign(W)
    \]
    where $\sign(W)$ is the ordinary signature of the $4k$-manifold $W$ and $\sign^{(2)}_\Gamma(W)$ is the $L^2$-signature of the intersection form
    \[
      H_{2k}(W;\cN\Gamma)\times H_{2k}(W;\cN\Gamma) \to \cN\Gamma
    \]
    over the group von Neumann algebra~$\cN\Gamma$ (e.g., see~\cite{Chang-Weinberger:2003-1}, \cite[Section~2.1]{Cha:2014-1}).
    We only need the fact that $|\sign^{(2)}_\Gamma(W)|$ is bounded by the $L^2$-dimension of the $\cN\Gamma$-module $H_{2k}(W;\cN\Gamma)$, which is in turn bounded by the number of $2k$-handles in a handle decomposition of~$W$.
    This should be viewed as the $L^2$ version of the same bound for the ordinary signature.
    It follows that $|\rho^{(2)}(M,i_G\circ\phi)|$ is at most twice the number of $2k$-handles of~$W$.
  \end{enumerate}

  Now, since $W$ has a PL triangulation with at most $C(n)\cdot \Delta(M)$ simplices, there is a handle decomposition of $W$ with at most $C'(n)\cdot \Delta(M)$ handles.
  By (1) and (2) above, it follows that
  \[
    |\rho^{(2)}(M,\phi)| = |\rho^{(2)}(M,i_G\circ\phi)| \le 2\cdot C'(n)\cdot \Delta(M).
    \qedhere
  \]
\end{proof}

\subsection{Applications to complexity estimates and stable complexity}

\begin{theorem}
  \label{theorem:complexity-additivity}
  Let $M$ be a closed orientable PL $n$-manifold with $n=4k-1$, and $\phi$ be an arbitrary group homomorphism of~$\pi_1(M)$.
  Then
  \[
    C(4k-1) \cdot r \cdot |\rho^{(2)}(M,\phi)| \le \Delta(\#^r M) \le r\cdot \Delta(M)
  \]
  where $C(4k-1)>0$ is a constant depending only on~$n=\dim M$.
\end{theorem}

\begin{proof}
  We have $\Delta(M\# N) \le \Delta(M) + \Delta(N)$ by constructing a triangulation of $M\# N$ from triangulations of $M$ and~$N$.
  From this, the right-hand side inequality follows.

  For the lower bound, we use the $\rho$-invariant as follows.
  Let
  \[
    \psi\colon \pi_1(\#^r M) = \coprod^r \pi_1(M) \to G
  \]
  be the homomorphism of the free product induced by ($r$ copies of) $\phi\colon \pi_1(M)\to G$.
  Then we have
  \begin{equation}
    \label{equation:rho-invariant-connected-sum}
    \rho^{(2)}(\#^r M,\psi) = r\cdot \rho^{(2)}(M,\phi).
  \end{equation}
  To see this, let $W$ be the bordism from the disjoint union $rM=\bigsqcup^r M$ to $\#^r M$ over $\Gamma$, which one obtains by attaching $r-1$ 1-handles to~$rM\times I$.
  Since there are no middle dimensional handles, both the ordinary signature $\sign(W)$ and the $L^2$-signatures $\sign^{(2)}_\Gamma(W)$ vanish.
  It follows that
  \[
    \rho^{(2)}(\#^r M,\psi) - \rho^{(2)}(rM, r\phi) = \sign^{(2)}_\Gamma(W) - \sign(W) = 0.
  \]
  By additivity under disjoint union, $\rho^{(2)}(rM, r\phi) = r \cdot \rho^{(2)}(M,\phi)$.
  Thus~\eqref{equation:rho-invariant-connected-sum} holds as claimed.

  Now, apply Theorem~\ref{theorem:main-rho-invariant} to $(\#^r M, \psi)$, to obtain
  \[
    \Delta(\#^r M) \ge C(4k-1)\cdot |\rho^{(2)}(\#^r M,\psi)| = C(4k-1)\cdot r \cdot |\rho^{(2)}(M,\phi)|.
    \qedhere
  \]
\end{proof}

Recall that Theorem~\ref{theorem:sum-stable-complexity-bound} asserts the following:
if $M$ is a closed orientable PL $(4k-1)$-manifold, then the inequality
\[
  C(4k-1) \cdot |\rho^{(2)}(M,\phi)| \le \Delta^{st}(M) \le \Delta(M)
\]
holds for all homomorphisms $\phi$ of~$\pi_1(M)$.
Consequently, $\Delta^{\text{st}}(M)>0$ if $\rho^{(2)}(M,\phi)\ne 0$ for some~$\phi$.

\begin{proof}[Prof of Theorem~\ref{theorem:sum-stable-complexity-bound}]
  Recall that the sum-stable complexity $\Delta^{\text{st}}(M)$ is defined by
  \[
    \Delta^{\text{st}}(M) = \lim_{r\to \infty} \frac{\Delta(\#^r M)}{r}.
  \]
  From this and the inequalities in Theorem~\ref{theorem:complexity-additivity}, Theorem~\ref{theorem:sum-stable-complexity-bound} follows immediately.
\end{proof}

As an example, consider the lens space $L_N=L(N;1,\ldots,1)$ of dimension $4k-1$ ($k\ge 1$) with $\pi_1(L_N)=\Z_N$.
Suppose $N>2$.
Let $\id=\id_{\pi_1(L_N)}$ be the identity.
Then it is known that $\rho^{(2)}(L_N,\id) \ge C'(4k-1) \cdot \Delta(L_N)$ for some constant $C'(4k-1)>0$ depending only on $4k-1=\dim L_N$.
In fact, $\rho^{(2)}(L_N,\id) = \frac1N \sum_{k=1}^{N-1} \cot^{2k} (k\pi / N)$ by computations in~\cite{Atiyah-Bott:1968-1,Atiyah-Patodi-Singer:1975-2}.
From this it follows that $\rho^{(2)}(L_N,\id) \ge (\frac13)^{3k} N^{2k-1}$ for $N>2$.
In the first paragraph of~\cite[Proof of Theorem~1.3]{Lim-Weinberger:2023-1}, a triangulation for $L_N$ is constructed by induction on the dimension, to show that $\Delta(L_N)\le C'' \cdot N^{2k-1}$ for some constant $C''=C''(4k-1)$.

From this, we obtain the following corollary to Theorem~\ref{theorem:sum-stable-complexity-bound} which is stated in the introduction:
for all $N>2$,
\[
  C(4k-1)\cdot \Delta(L_N) \le \Delta^{\mathrm{st}}(L_N) \le \Delta(L_N)
\]
where $C(4k-1)>0$ is a constant depending only on $4k-1=\dim L_N$.

\begin{remark}
  An open problem that appears interesting on its own is whether the value of $\Delta^{\mathrm{st}}(M)$ is always a rational number.
  We do not address it in this paper.
\end{remark}

\subsection{Unbounded complexity within a fixed simple homotopy type}

\begin{proof}[Proof of Theorem~\ref{theorem:unbounded-complexity-fixed-homotopy-type}]
  Let $M$ be a closed orientable $(4k-1)$-manifold with an element $g$ of order $d$ in~$\pi_1(M)$, $1<d<\infty$.
  Our goal is to obtain a $(4k-1)$-manifold $N$ which is simple homotopy equivalent to $M$ and has arbitrarily large complexity.

  We use the argument of the proof of~\cite[Proof of Theorem~1]{Chang-Weinberger:2003-1}, combining it with Theorem~\ref{theorem:main-rho-invariant}, as spelled out below.
  The inclusion $i\colon \Z_d \hookrightarrow \pi_1(M)$ sending $1$ to $g$ induces a homomorphism $i_*\colon L^s_{4k}(\Z[\Z_d]) \to L^s_{4k}(\Z[\pi_1(M)])$ on the $L$-groups.
  Let $\alpha$ be a form over $\Z[\Z_d]$ representing an element in $L^s_{4k}(\Z[\Z_d])$.
  The Wall realization gives a bordism $W$, over $M$, from $(M,\id_M)$ to $(N,f)$, where $N$ is a $(4k-1)$-manifold equipped with a simple homotopy equivalence $f\colon N\to M$, such that the intersection form of $W$ is $i_* \alpha$.
  Let $f_*\colon \pi_1(N) \to \pi_1(M)$ be the induced homomorphism.
  Then we have
  \[
    \rho^{(2)}(N,f_*) - \rho^{(2)}(M,\id_{\pi_1(M)}) = \sign^{(2)}_{\pi_1(M)}(i_* \alpha) = \sign^{(2)}_{\Z_d}(\alpha)
  \]
  by the $L^2$-signature defect formula and $L^2$-induction.
  From this and Theorem~\ref{theorem:main-rho-invariant}, it follows that the complexity of $N$ is large if $\sign^{(2)}_{\Z_d}(\alpha)$ is large.
  So the desired conclusion follows if $\smash{\sign^{(2)}_{\Z_d}} \colon L^s_{4k}(\Z[\Z_d]) \to \R$ has unbounded image.
  Since $\smash{\sign^{(2)}_{\Z_d}}$ is additive under direct sum and since $L^s_{4k}(\Z[\Z_d])$ is isomorphic, modulo torsion, to the $L$-group of symmetric forms on projective $\Q[\Z_d]$-modules, it suffices to observe that the $1\times 1$ form $\beta=[1]$ on the projective $\Q[\Z_d]$-module~$\Q$ satisfies $\smash{\sign^{(2)}_{\Z_d}}(\beta) = \frac1d - 1 \ne 0$.
  Here we use that the $L^2$-dimension over $\cN \Z_d = \C[\Z_d]$ is equal to $\frac1d \cdot \dim_\C$.
\end{proof}

\section{Quantitative bordism for 1-manifolds}
\label{section:questions-for-n=1}

Question~\ref{question:linear-bordism-over-acyclic} asks whether every closed PL $n$-manifold $M$ over an acyclic group~$G$ admits a PL bordism $W$ from $M$ to a trivial end over $G$ with linearly bounded complexity $\Delta(W) \le C(n)\cdot \Delta(M)$, for $n > 1$.
In what follows, we consider the case of $n=1$.

\begin{proposition}
  \label{proposition:questions-for-n<=2}
  The answer to the oriented case of Question~\ref{question:linear-bordism-over-acyclic} is negative for $n=1$.
\end{proposition}

\begin{proof}
  The following approach, which uses the commutator length in a group, was suggested by Shmuel Weinberger.
  Let $f\colon S^1\to BG$ be a map.
  Let $\alpha\in G$ be the image of a generator of $\pi_1(S^1)$ under $f_*\colon \pi_1(S^1) \to \pi_1(BG)$.
  If $S^1$ bounds an oriented surface of genus $g$ over $BG$, then $\alpha$ is a product of $g$ commutators $[a_i,b_i]$, $a_i,b_i\in G$.
  That is, the commutator length
  \[
    \cl(\alpha) = \min\{k\ge 0 \mid \alpha \text{ is expressed as a product of $k$ commutators}\}
  \]
  is at most~$g$.
  Note that the complexity of a surface is bounded from below by the genus.
  Also, if there is a bordism $W$ from $f\colon S^1\to BG$ to a trivial end with complexity $\Delta(W)$, then there is a null-bordism for $f$ with complexity${}\le 2\cdot \Delta(W)$, by capping off boundary components other than~$f$.
  It follows that $\cl(\alpha)/2$ is a lower bound for~$\Delta(W)$.

  If $G$ is a free group $F$ of rank${}>1$, there exists an element $g$ in $[F,F]$ with arbitrarily large~$\cl(g)$.
  (For instance, $\cl([a,b]^k)=\lfloor k/2 \rfloor + 1$ in $F=F\langle a,b\rangle$~\cite{Culler:1981-1}.)
  It follows that Question~\ref{question:linear-bordism}, which is a stronger version of Question~\ref{question:linear-bordism-over-acyclic}, has a negative answer for $n=1$.

  Moreover, there is an acyclic group that has an element with arbitrarily large commutator length.
  For instance, there exists a nontrivial torsion-free acyclic hyperbolic group~$G$.%
  \footnote{
    Following Wilton and Lyman~\cite{Wilton-Lyman:2020-1}, $G = \langle a,b \mid a[a,b]^5, b^6a^2b^3a^{-1}b^{-8}a^{-1}\rangle$ is an example:
    it is perfect and satisfies the $C'(1/6)$ small cancellation condition.
    By small cancellation theory, $G$ is hyperbolic and the 2-complex $X$ associated with the presentation is aspherical, and thus $G$ is torsion-free.
    Also, $G$ is acyclic since $\tilde H_*(X)=0$.
  }
  In $G$, there exists an element $\alpha$ for which the stable commutator length $\scl(\alpha) = \lim_{r\to\infty} \cl(\alpha^r)/r$ is nonzero~\cite[p.~67]{Calegari:2009-1}, and consequently $\operatorname{cl}(\alpha^r)$ is unbounded.
  Therefore, an acyclic group counterexample is obtained by applying the same argument.
  This shows that the answer to Question~\ref{question:linear-bordism-over-acyclic} is negative for $n=1$.
\end{proof}

We remark that the same argument also proves a negative answer to the unoriented case of Question~\ref{question:linear-bordism} for $n=1$:
one obtains an explicit counterexample over the free groups $F=F\langle a_1,\ldots a_k\rangle$ of rank~$2k$, using the fact that $a_1^2\cdots a_k^2$ is not expressed as a product of fewer than $k$ squares in $F$~\cite{Lyndon-Newman:1973-1}.

\bibliographystyle{amsalpha-order}
\def\MR#1{}
\bibliography{research.bib}

\end{document}